\documentclass[11pt]{article}\textwidth 160mm\textheight 235mm
\usepackage{amsmath}
\usepackage{amsthm}
\usepackage{amsfonts}
\usepackage{multirow}
\usepackage[table,xcdraw]{xcolor}
\usepackage{subcaption}
\usepackage{enumerate}
\usepackage{float}
\usepackage{algorithm}
\usepackage{algpseudocode}
\usepackage{hyperref}
\usepackage{tabularx}
\usepackage{amssymb}
\usepackage{graphicx}
\usepackage{tcolorbox}
\usepackage{booktabs}
\usepackage[table,xcdraw]{xcolor}
\usepackage{multirow}
\usepackage{rotating,tabularx}

\usepackage[left=2cm,right=2cm,top=2cm,bottom=2cm]{geometry}
\newcolumntype{L}[1]{>{\raggedright\let\newline\\\arraybackslash\hspace{0pt}}m{#1}}
\newcolumntype{C}[1]{>{\centering\let\newline\\\arraybackslash\hspace{0pt}}m{#1}}
\newcolumntype{R}[1]{>{\raggedleft\let\newline\\\arraybackslash\hspace{0pt}}m{#1}}

\newtheorem{Theorem}{Theorem}[section]
\newtheorem{Proposition}[Theorem]{Proposition}
\newtheorem{Remark}[Theorem]{Remark}
\newtheorem{Lemma}[Theorem]{Lemma}
\newtheorem{Corollary}[Theorem]{Corollary}
\newtheorem{Definition}[Theorem]{Definition}
\newtheorem{Example}[Theorem]{Example}

\usepackage{hyperref}
\hypersetup{colorlinks=true,urlcolor=blue}
\expandafter\let\expandafter\oldproof\csname\string\proof\endcsname
\let\oldendproof\endproof
\renewenvironment{proof}[1][\proofname]{
\oldproof[\ttfamily\scshape \bf #1.]
}{\oldendproof}
\def\ve{\varepsilon}

\def\tilde{\widetilde}
\def\emp{\emptyset}

\def\dom{{\rm dom}\,}

\def\epi{{\rm epi\,}}
\def\rge{{\rm rge\,}}

\def\ox{\overline{x}}
\def\oy{\overline{y}}
\def\oz{\overline{z}}

\def\disp{\displaystyle}
\def\tto{\rightrightarrows}

\def\Bar{\overline}
\def\ra{\rangle}
\def\la{\langle}
\def\ve{\varepsilon}
\def\epsilon{\varepsilon}
\def\ox{\bar{x}}
\def\oy{\bar{y}}
\def\oz{\bar{z}}

\def\ov{\bar{v}}

\def\rri{\rightrightarrows}

\def\gph{\mbox{\rm gph}\,}
\def\epi{\mbox{\rm epi}\,}

\def\dom{\mbox{\rm dom}\,}

\def\argmin{\mathop{{\rm argmin}}}

\def\dn{\downarrow}
\def\O{\Omega}

\def\emp{\emptyset}

\def\oR{\Bar{\R}}

\def \N{{\rm I\!N}}
\def \R{{\rm I\!R}}

\def\Limsup{\mathop{{\rm Lim}\,{\rm sup}}}
\def\Liminf{\mathop{{\rm Lim}\,{\rm inf}}}
\def\Lim{\mathop{{\rm Lim}}}

\def\Limsup{\mathop{{\rm Lim}\,{\rm sup}}}

\numberwithin{equation}{section}

\def\gliminf{\mathop{{\rm g} \text{-}\,{\rm liminf}}}
\def\glimsup{\mathop{{\rm g} \text{-}\,{\rm limsup}}}
\def\eliminf{\mathop{{\rm e} \text{-}\,{\rm liminf}}}
\def\elimsup{\mathop{{\rm e} \text{-}\,{\rm limsup}}}
\def\elim{\mathop{{\rm e} \text{-}\,{\rm lim}}}
\def\glim{\mathop{{\rm g} \text{-}\,{\rm lim}}}
\def\qua{\mathrm{quad}}
\def\sgn{\mathrm{sgn}}

\numberwithin{equation}{section}
\title{\bf
Generalized Twice Differentiability and Quadratic Bundles in Second-Order Variational Analysis}

\author{Pham Duy Khanh\footnote{Group of Analysis and Applied Mathematics, Department of Mathematics, Ho Chi Minh City University of Education, Ho Chi Minh City, Vietnam. E-mail: khanhpd@hcmue.edu.vn.} \quad Boris S. Mordukhovich\footnote{Department of Mathematics, Wayne State University, Detroit, Michigan, USA. E-mail: aa1086@wayne.edu. Research of this author was partly supported by the US National Science Foundation under grant DMS-2204519, by the Australian Research Council under Discovery Project DP-190100555, and by Project 111 of China under grant D21024.}\quad Vo Thanh Phat\footnote{Department of Mathematics  and Statistics, University of North Dakota, Grand Forks, North Dakota, USA. E-mail: thanh.vo.1@und.edu.} \quad Le Duc Viet\footnote{Department of Mathematics, Wayne State University, Detroit, Michigan, USA. E-mail: vietle@wayne.edu. Research of this author was partly supported by the US National Science Foundation under grant DMS-2204519.}}

\begin{document}
\maketitle

\noindent
{\small{\bf Abstract}. In this paper, we investigate the concepts of generalized twice differentiability and quadratic bundles of nonsmooth functions that have been very recently proposed by Rockafellar in the framework of second-order variational analysis. These constructions, in contrast to second-order subdifferentials, are defined in primal spaces. We develop new techniques to study generalized twice differentiability for a broad class of prox-regular functions, establish their novel characterizations. Subsequently, quadratic bundles of prox-regular functions are shown to be nonempty, which provides the ground of potential applications in variational analysis and optimization.\\[1ex]
{\bf Mathematics Subject Classification (2020)}} 49J52, 49J53, 90C31

\section{Introduction}\label{intro}
\vspace*{-0.1in}

Starting from the original works \cite{2sd2,2sd}, {\em second-order epi-derivatives/subderivatives}  have been designed as neoclassical counterparts of the Hessians, while concerning \textit{epigraphs} of second-order difference quotients in contrast to their \textit{graph} in the classical approach. The corresponding property, labeled as \textit{twice epi-differentiability}, holds in broad settings such as maxima of $\mathcal{C}^2$-smooth functions, compositions of piecewise linear-quadratic and smooth functions  \cite{Rockafellar98},  and---more generally---{\em parabolically regular} functions; see \cite{mms,sarabi}. However, generalized second-order derivatives of this {\em primal type} suffer from the {\em lack of robustness}, which prevents developing satisfactory {\em calculus rules} and broad applications. This is contrary to the {\em second-order subdifferential} by Mordukhovich, a robust construction of the {\em dual type}; see \cite{Mordukhovich24} and the references therein. In particular, second-order subderivatives are capable of characterizing \textit{strong local minimizers} yet fail to do so for \textit{tilt-stable local minimizers}, a variational stability property requiring stable behavior of local minimizers under tilted/linear shifts of objective functions; see the landmark work \cite{tilt} and many subsequent publications as, e.g., \cite{drus1,gferer,Mordukhovich24,tiltnghia,tiltrocka} among others.
\vspace*{0.03in}

This calls for the need of more robust generalized second-order constructions of primal and/or primal-dual types
to deal with robust stability and numerical issues of variational analysis and optimization. Quite recently \cite{roc}, Rockafellar introduces novel generalized second-order derivatives, called \textit{quadratic bundles}, which are defined as collections of \textit{epigraphical limits} of sequences of second-order subderivatives of primal-dual pairs converging to the point in question, with the objective function being {\em generalized twice differentiable} along these sequences. His motivations come from the study of  \textit{variational sufficiency} in problems of parametric composite optimization with applications to convergence of numerical algorithms of proximal and augmented Lagrangian types.\vspace*{0.03in}

In this paper, we explore the concept of quadratic bundles as the basis for broad upcoming applications. Our first goal is to provide a systematic study of {\em generalized twice differentiability}, a concept that requires both \textit{twice epi-differentiability} and the derivative itself being a \textit{generalized quadratic form}, which serves as a crucial building block for our upcoming endeavors. Although the fundamental papers \cite{genhes,2sd2} provide much insight into these two concepts, the precise formulation and name are coined recently in \cite{roc}. We now systematically investigate quadratic bundles, with a revision of this construction that operates properly in the absence of the subdifferential continuity inherent in convex functions. As the cumulation of the study below, {\em prox-regular functions} are shown to always have {\em nonempty quadratic bundles}, which testifies to the importance and usefulness of this concept.\vspace*{0.03in}

To achieve our goals in the general setting of prox-regular functions, the notions of \textit{Moreau envelopes} and associated \textit{proximal mappings} are of paramount importance. Originally introduced by Moreau in \cite{moreau62}, proximal mappings of convex functions were viewed as generalizations of projections. The associated epi-addition (infimal convolution) operation in Moreau envelopes has been also implemented but not yet as an approximation and regularization tool for the function in question. To the best of our knowledge, this crucial interpretation appeared for the first time in the subsequent works by Attouch and Wets for convex and nonconvex functions \cite{attouch84,attouchwets83}. Moreau envelopes play a significant role in our developments below.\vspace*{0.03in}

In \cite{proxreg}, Poliquin and Rockafellar introduce the class of extended-real-valued \textit{prox-regular functions}, which by now has been recognized as the main class of functions in second-order variational analysis. This class encompasses, in particular, $\mathcal{C}^{1,1}$-smooth functions and its lower-$\mathcal{C}^2$ counterparts, strongly amenable functions, convex and variationally ones, etc. It is verified in the same paper \cite{proxreg} that Moreau envelopes of prox-regular functions are (under mild additional assumptions) of class $\mathcal{C}^{1,1}$ around the reference point, i.e., ${\cal C}^1$-smooth with locally continuous gradients. The recent paper \cite{roc} reveals certain connections between the newly introduced quadratic bundles of functions and the  {\em Hessian bundles} of the associated Moreau envelopes, which are conveniently guaranteed to be nonempty with the $\mathcal{C}^{1,1}$-smoothness of the envelopes. These facts and observations confirm the irreplaceable role of Moreau envelopes and the underlying assumption of prox-regularity, which we employ and develop throughout our work.\vspace*{0.03in}

The rest of the paper is organized as follows. Section~\ref{sec:prel} equips the reader with some basic material of variational analysis and generalized differentiation that is broadly exploit in deriving the major results below. Section~\ref{sec:varsconv} revisits the very recent unification of variational convexity and prox-regularity and their connection with second-order subderivatives while deriving some new facts in this direction. Section~\ref{sec:gtd} addresses generalized twice differentiable functions with establishing their important properties. This study is continued in Section~\ref{gen-moreau} by involving prox-regularity and Moreau envelopes. Here we obtain the crucial characterization of generalized twice differentiability for prox-regular functions via classical twice differentiability of Moreau envelopes. The obtained result serves as the central building block for the subsequent sections. Section~\ref{sec:quad} offers a revised version of quadratic bundles for nonconvex functions and presents new facts and instructive examples for this second-order generalized derivative construction. Among the most significant results of this section is establishing the nonemptiness of quadratic bundles associated with any extended-real-valued prox-regular function. Section \ref{sec:conclusion} recaps the major ideas and results achieved in the paper lists some directions of our future research.

\section{Preliminaries from Variational Analysis}\label{sec:prel}

Throughout the paper, we use the standard terminology and notation of variational analysis and generalized differentiation; see \cite{Mordukhovich06,Mordukhovich18,Rockafellar98}. Recall that $\R^n$ stands for the $n$-dimensional space with the inner product $\la \cdot, \cdot \ra$ and the Euclidean (unless otherwise stated) norm (norm-2) written as  $\|\cdot\|$. The symbol $B(\bar{x},r)$ signifies the open ball centered at $\bar{x}$ with radius $r > 0$. We write $x_k\xrightarrow{C}\bar{x}$ when $x_k\to\ox$ as $k\to\infty$ with $x_k\in C$ for all $k\in\N:=\{1,2,\ldots\}$.  For a (linear) subspace $L\subset\R^n$, recall the {\em Pythagorean formula}
\begin{equation}\label{eq:pythagoras}\|w\|^2 = \|P_Lw\|^2 + \|P_{L^\perp}w\|^2\;\mbox{ whenever }\;w\in\R^n,
\end{equation}
where the {\em orthogonal projector} $P_L$ is the self-adjoint linear operator whose range is $L$ and whose kernel is $L^\perp:=\{w\in\R^n\;|\;\la w,x\ra=0,\;x\in L\big\}$.\vspace*{0.03in}

Given now an extended-real-value function $f: \R^n \to \R \cup \{\infty\} := \overline{\R}$, the \textit{domain} and \textit{epigraph} of $f$ are defined, respectively, by 
$$
\dom f :=\big\{x \in \R^n\;\big|\;f(x) < \infty\big\},\quad\epi f := \big\{(x,\alpha) \in \R^n \times \R\;\big|\;\alpha \ge f(x)\big\}. 
$$
The {\em properness} of $f$ means that $\dom f\ne\emp$, which we always assume. Recall that $f$ is {\em lower semicontinuous} (l.s.c.) at $\ox$ if $\liminf\limits_{x \to \bar{x}} f(x) \ge f(\bar{x})$. This property holds on a set $C$ if $f$ is l.s.c.\ at any $\ox\in C$.\vspace*{0.03in}

The (general/basic, limiting, Mordukhovich) {\em subdifferential} of $f\colon\R^n\to\oR$ at $\ox\in\dom f$ is defined by
\begin{equation}\label{eq:subdiff}
\partial f(\bar{x}):=\Big\{v\in\R^n\;\Big|\;\exists\,x_k\to\ox,\;v_k\to v\;\mbox{ with }\;\disp\liminf_{x\to x_k} \frac{f(x)-f(x_k)-\la v_k,x-x_k\ra}{\|x-x_k\|}\ge 0\Big\}.
\end{equation}
When $f$ is convex, \eqref{eq:subdiff} reduces to the classical subdifferential of convex analysis, while otherwise $\partial f(\ox)$ is often nonconvex; e.g., we have $\partial f(0)=\{-1,1\}$ for the function $f(x)=-|x|$, $x\in\R$. Nevertheless, the subdifferential $\partial f$ satisfies comprehensive calculus rules for the general class of l.s.c.\ functions. This calculus is based on variational/extremal principles and techniques of variational analysis; see \cite{Mordukhovich06,Mordukhovich18,Rockafellar98}.\vspace*{0.03in}

Let $F:\R^n \rightrightarrows \R^m$ be a set-valued mapping/multifunction with the {\em graph} $\gph F :=\{(x,y) \in \R^n \times \R^m\;|\;y \in F(x)\}$. The {\em graphical derivative} of $F$ at $(\ox,\oy)\in\gph F$ is
\begin{equation}\label{eq:graphderivative}
DF(\bar{x},\bar{y})(u): = \big\{v \in \R^n\big|\;(u,v) \in T_{\gph F}(\bar{x},\bar{y})\big\},\quad u\in\R^n,
\end{equation}
where the {\em tangent/contingent cone} to a set $\O\subset\R^d$ at $\oz\in\O$ is defined by
\begin{equation*}
T_\O(\bar{z}) :=\big\{ w\in\R^d\;\big|\;\exists\,z_k\xrightarrow{\O} \bar{x},\; t_k \downarrow 0\;
\mbox{ with }\;(x_k - \bar{x})/t_k \to w\big\}
\end{equation*}

Next we discuss the concept of \textit{epi-convergence} for sequences of functions, which plays a crucial role in defining first-order and second-order subderivatives; more details can be found in \cite{attouch84,Rockafellar98}. To begin with, recall {\em the outer limit} and {\em inner limit} of the set sequence $\{C_k\}\subset\R^n$ defined, respectively, by
$$ 
\Limsup\limits_{k \to \infty} C_k: =\big\{ x\;\big|\;\exists\,x_k \in C_k \;\text{ and a subsequence }\; 
\{x_{k_j}\}\;\mbox{ with }\;x_{k_j} \to x\;\mbox{ as }\;j\to\infty\big\},$$
$$
\Liminf\limits_{k \to \infty} C_k:=\big\{x\in\R^n\;\big|\;\exists\,x_k \in C_k\;\mbox{ with }\;x_k \to x\;\mbox{ as }\;k\to\infty\big\}.
$$
It is said that $\{C_k\}$ is {\em convergent} to some $C$ (denoted as $C_k\to C$) if 
$$
\Limsup\limits_{x \to \infty} C_k = \Liminf\limits_{x \to \infty} C_k=:\Lim\limits_{x \to \infty}C_k=C.
$$ 
When $F: \R^n \to \R^m$ is \textit{norm-coercive}, in the sense that $\|F(x)\| \to \infty$ as $\|x\| \to \infty$, the convergence $C_k \to C$ implies  that $F(C_k) \to F(C)$; see \cite[Theorem~4.26]{Rockafellar98}.\vspace*{0.03in}

Now we recall epigraphical limits for sequences of functions and graphical limits for sequences of multifunctions by following the book \cite{Rockafellar98}. For a sequence of functions $f_k : \R^n \to \overline{\R}$, the \textit{lower epigraphical limit} $\eliminf\limits_{k \to \infty} f_k$ and the \textit{upper epigraphical limit} $\elimsup\limits_{k \to \infty} f_k$ are defined via their epigraphs as 
$$
\epi\Big( \eliminf\limits_{k \to \infty} f_k\Big) = \Limsup\limits_{k \to \infty} (\epi f_k),\quad
\epi\Big( \elimsup\limits_{k \to \infty} f_k\Big) = \Liminf\limits_{k \to \infty} (\epi f_k).
$$
When these limits agree, we say that $(f_k)$ \textit{epigraphically converges} to $f$ and  denote this by either $\elim\limits_{k \to \infty} f_k = f$, or $f_k \xrightarrow{e} f$. If in addition $-f_k \xrightarrow{e} -f$, then $f_k$ \textit{continuously converges} to $f$, which is denoted by $f_k \xrightarrow{c} f$ and means that for all $x_k \to x$ and $x\in\R^n$, we have $f_k(x_k) \to f(x)$; see \cite[Theorem~7.11]{Rockafellar98}. An alternative statement to determine the lower and upper epigraphical limits together with a characterization of epi-convergence is given in the following result taken from \cite[Proposition~7.2]{Rockafellar98}.

\begin{Proposition}\label{prop:epichar} Let $\{f_k\}$ be a sequence of functions on $\R^n$, and let $x\in\R^n$. Then we have
$$
\Big(\eliminf\limits_{k \to \infty} f_k\Big)(x) = \min\big\{ \alpha \in \overline{\mathbb{\R}}\;\big|\;\exists\, x_k \to x,\;\liminf f_k(x_k) = \alpha\big\}, 
$$
$$
\Big(\elimsup\limits_{k \to \infty} f_k\Big)(x) = \min\big\{ \alpha \in \overline{\mathbb{\R}}\;\big|\;\exists\,x_k \to x,\;\limsup f_k(x_k) = \alpha\big\}. 
$$
Therefore, $\{f_k\}$ epigraphically converges to $f$ if and only if at each point $x$ it holds
$$
\left\{ \begin{aligned}
&\liminf f_k(x_k) \ge f(x) \text{ for every sequence } x_k \to x,\\
&\limsup f_k(x_k) \le f(x) \text{ for some sequence } x_k \to x.
\end{aligned} \right. $$
In particular, if $\{f_k\}$ epigraphically converges to $f$, then for all $x$ there exists a sequence $x_k \to x$ such that $f_k(x_k) \to f(x)$ as $k\to\infty$.
\end{Proposition}

Epi-convergence is instrumental for the introduction of second-order subderivatives of extended-real-valued functions; see \cite[Definition~13.3]{Rockafellar98}. Consider first the {\em second-order difference quotients} of $f\colon\R^n\to\oR$ at $\ox\in\dom f$ for $\ov\in\R^n$ given as 
\begin{equation*}
\Delta_t^2 f(\bar{x}|\bar{v})(w) := \frac{f(\bar{x} + t w) - f(\bar{x}) - t \la \bar{v}, w \ra}{\frac{1}{2}t^2}\;\mbox{ whenever }\;w\in\R^n\;\mbox{ and }\;t>0
\end{equation*}
and then define and {\em second-order subderivative} of $f$ at $\bar{x}$ for $\bar{v}$ by
\begin{equation}\label{eq:d2}
d^2 f(\bar{x}|\bar{v})(w) := \Big(\eliminf\limits_{t \downarrow 0} \Delta_t^2 f(\bar{x}|\bar{v})\Big)(w) = \liminf\limits_{\substack{t \downarrow 0 \\ w' \to w}} \Delta_t^2 f(\bar{x}|\bar{v})(w'), \quad w\in \R^n.
\end{equation}
Note that $d^2 f(\bar{x}|\bar{v})$ may take the value $-\infty$. When the epigraphical lower limit in \eqref{eq:d2} is actually a full limit, we say that  of $f$ is \textit{twice epi-differentiable }at $\bar{x}$ for $\bar{v}$. If in addition $d^2 f(\bar{x}|\bar{v})$ is proper, then $f$ is \textit{properly twice epi-differentiable} at $\bar{x}$ for $\bar{v}$. A major subclass of nonsmooth twice epi-differentiable functions consists of {\em parabolically regular} ones; see \cite{mms,sarabi,Rockafellar98}. Observe that a similar formula to \eqref{eq:d2} for the upper epigraphical limit of $\Delta_t^2 f(\bar{x}|\bar{v})$ as $t \downarrow 0$ is not available, and we usually have to appeal to 
Proposition~\ref{prop:epichar} providing the expression
$$
\Big(\elimsup\limits_{t \downarrow 0} \Delta_t^2 f(\bar{x}|\bar{v})\Big)(w)  = \min\big\{ \alpha \in \overline{\mathbb{\R}}\;\big|\;\exists\, t_k \downarrow 0, w_k \to w,\;\limsup \Delta_{t_k}^2 f(\bar{x}|\bar{v})(w_k) = \alpha\big\}
$$
when calculating this part of twice epi-differentiability, see, e.g., the function in Example~\ref{rmk:C1gtd} below. In the case of classical twice differentiability, second-order subderivatives reduce to quadratic forms associated with the {\em Hessian matrices} at $\bar{x}$ (see \cite[Proposition~13.8]{Rockafellar98}):
$$
d^2 f(\bar{x}|\nabla f(\bar{x}))(w) = \la w, \nabla^2 f(\bar{x})w \ra, \quad w \in \R^n.
$$
To get $d^2f(\bar{x}|\bar{v})(w)\in\R$, it is necessary that $w$ belongs to the {\em critical cone} of $f$ at $\bar{x}$ for $\bar{v}$ defined by
\begin{equation}\label{eq:critcone}
K(\bar{x},\bar{v}) :=\big\{ w\in\R^n\;\big|\;df(\bar{x})(w) = \la \bar{v},w \ra\big\},
\end{equation}
where the {\em first-order subderivative} of $f$ at $\ox\in\dom f$ is given by
\begin{equation}\label{1sub}
df(\ox)(w):=\liminf\limits_{\substack{t \downarrow 0 \\ w' \to w}} \Delta_t f(\ox)(w')\;\mbox{ with }\;\Delta_t f(\ox)(w'):=\frac{f(\ox+t w')-f(\ox)}{t},\quad w\in\R^n.    
\end{equation}

Next we consider set-valued mappings $F_k: \R^n\tto\R^m$, $k\in\N$, and define for them the {\em graphical outer limit} and {\em graphical inner limit} via the set convergence
$$
\gph\Big(\glimsup\limits_{k \to \infty} F_k\Big) := \Limsup\limits_{k \to \infty} (\gph F_k),\quad
\gph\Big( \gliminf\limits_{k \to \infty} F_k\Big) := \Liminf\limits_{k \to \infty} (\gph F_k). 
$$
When these limits agree, we say that $\{F_k\}$ {\em graphically converges} to some $F$  denoted as $F = \glim\limits_{k \to \infty} F_k$ or $F_k \xrightarrow{g} F$. The usage of graphical convergence allows us represent the graphical derivative \eqref{eq:graphderivative} as 
\begin{equation}\label{eq:8(15)}
DF(\bar{x}|\bar{y}) = \glimsup\limits_{t \downarrow 0} \Delta_t F(\bar{x}|\bar{y}),\;\mbox{ where }\;\Delta_t F(\bar{x}|\bar{y})(w) := \frac{F(\bar{x} + tw) - \bar{y}}{t}.
\end{equation}
A multifunction $F$ is {\em proto-differentiable} at $\bar{x}$ for $\bar{y} \in F(\bar{x})$ if the limit in \eqref{eq:8(15)} is a full graphical limit.

\begin{Proposition}\label{rmk:protosingle_diff} If a single-valued mapping $F: \R^n \to \R^m$ is differentiable at $\bar{x} \in \R^n$, then $F$ is proto-differentiable at $\bar{x}$ for $\bar{y}= F(\bar{x})$.
\end{Proposition}
\begin{proof} Assume that $F: \R^n \to \R^m$ is differentiable at $\bar{x} \in \R^n$. Then we get by \cite[Proposition~8.34]{Rockafellar98} that $DF(\bar{x}|\bar{y}) = \nabla F(\bar{x})$.  This gives us 
$\bar{z} = \nabla F(\bar{x})\bar{w}$ for any $\bar{w} \in \R^n$ and $\bar{z} \in DF(\bar{x}|\bar{y})(\bar{w})$. By the differentiability of $F$ at $\bar{x}$, for any $t_k \downarrow 0$ we have
$$
\lim\limits_{k \to \infty}\frac{F(\bar{x} + t_k \bar{w}) - F(\bar{x})}{t_k} = \nabla F(\bar{x})\bar{w}.
$$
Letting $z_k := \dfrac{F(\bar{x} + t_k \bar{w}) - F(\bar{x})}{t_k}$ and $w_k := \bar{w}$ verifies the proto-differentiability of $F$ at $\bar{x}$ for $\bar{y}$.
\end{proof}

Finally in this section, we recall the notions of {\em Moreau $\lambda$-envelopes} and {\em $\lambda$-proximal mappings} associated with proper l.s.c.\ functions $f$ with a parameter $\lambda > 0$ defined by
$$
e_\lambda f(x) := \inf\limits_{u \in \R^n}\Big\{ f(u) + \frac{1}{2\lambda}\|u - x\|^2\Big\},\quad P_\lambda f(x) := \argmin\limits_{u \in \R^n} \big\{ f(u) + \frac{1}{2\lambda}\|u - x\|^2\}.
$$
The function $f$ is \textit{prox-bounded} if there exists $\lambda > 0$ such that $e_{\lambda} f(x) > -\infty$ for some $x \in \R^n$, which is equivalent to saying that $f$ is minorized by a quadratic function.

\section{Variational $s$-Convexity and Second-Order Subderivatives}\label{sec:varsconv}

In the recent paper \cite{varprox}, Rockafellar unifies the {\em variational convexity} and {\em prox-regularity} of extended-real-functions on finite-dimensional spaces into a single concept called {\em variational $s$-convexity}. Recall first that  $f: \R^n \to \overline{\R}$ is {\em $s$-convex} on $\R^n$ for a given number $s \in \R$ if its quadratic shift $f - \frac{s}{2}\|\cdot\|^2$ is convex on the entire space. This global property is called {\em $s$-strong convexity} when $s > 0$ and {\em $(-s)$-weak convexity} when $s < 0$. 
Now we formulate the major notion.

\begin{Definition}\label{defi:varconv} {\rm Let $s$ be any real number. A proper l.s.c.\ function $f\colon\R^n \to \overline{\R}$ is {\em variationally $s$-convex} at $\bar{x}$ for $\bar{v} \in \partial f(\bar{x})$ if $f$ is finite at $\bar{x}$ and there exist an $s$-convex function $\widehat{f}$ as well as convex neighborhoods $U$ of $\bar{x}$ and  $V$ of $\bar{v}$ together with $\varepsilon > 0$ such that $\widehat{f} \le f$ on $U$ and 
$$
(U \times V) \cap \gph \partial \widehat{f} = (U_\varepsilon \times V) \cap \gph \partial f, \quad\widehat{f}(x) = f(x)\;\text{ at common elements }\;(x,v), 
$$
where $U_\varepsilon := \{x \in U \mid f(x) < f(\bar{x}) + \varepsilon\}$. We can always select $U := B(\bar{x},\varepsilon)$, $V := B(\bar{v},\varepsilon)$ and then say that $f$ is variationally $s$-convex at $\bar{x}$ for $\bar{v}$ \textit{with corresponding radius} $\varepsilon > 0$.} When $s \le 0$, $f$ is said to be $(-s)$-level {\em prox-regular} at $\bar{x}$ for $\bar{v}$, while for $s = 0$ we say that $f$ is {\em variationally convex}  at $\bar{x}$ for $\bar{v}$. When $s > 0$, $f$ is called to be {\em $s$-strongly variationally convex} at $\bar{x}$ for $\bar{v}$.
\end{Definition}

Prox-regular functions have been around for a long time; see Section~\ref{intro} and also the books \cite{Mordukhovich24,thibault} in infinite dimensions. On the other hand, the notion of variational convexity appears rather recently in \cite{varconv} in finite-dimensional spaces while being further studied, characterized, and applied in this framework in \cite{var2order,kmp24,varsuff1,Mordukhovich24,roc}. Infinite-dimensional developments are presented in \cite{varconvinf}. The novel unified version from Definition~\ref{defi:varconv} is utilized in \cite{gfrerer24,varprox,roc24}.\vspace*{0.03in}

Let us now discuss the notion of $f$-attentive $\ve$-localizations of $\partial f$ around a point. Given an l.s.c.\ function $f: \R^n \to \overline{\R}$, its subdifferential \eqref{eq:subdiff}, and some number  $\varepsilon > 0$, we say that a set-valued mapping $T_\ve: \R^n \rri \R^n$ is an {\em $f$-attentive $\varepsilon$-localization} of $\partial f$ around $(\bar{x},\bar{v}) \in \gph \partial f$ if
\begin{equation}\label{eq:fattentive}
\gph T_\ve = \big\{ (x,v) \in \gph \partial f\;\big|\;x \in B(\bar{x},\varepsilon),\;v \in B(\bar{v},\varepsilon),\;f(x) < f(\bar{x}) + \varepsilon\big\}.
\end{equation}

By Definition~\ref{defi:varconv}, an $r$-level prox-regular function $f$ at $\bar{x}$ for $\bar{v}$ agrees with an $r$-weakly convex function $\widehat{f}$ on an $f$-attentive localization of $\partial f$ around $(\bar{x},\bar{v})$, while an ($s$-strongly) variational convex function coincides with an ($s$-strongly) convex function in the similar manner. In light of this interpretation, we get the following characterization of variationally $s$-convex functions $f$ taken from \cite[Theorem~1]{varprox}.

\begin{Theorem}\label{theo:svarchar} Given level $s \in \R$ and an extended-real-valued function $f: \R^n \to \overline{\R}$, suppose that $f$ is finite and l.s.c.\ at $\bar{x}$. Then $f$ is variationally $s$-convex at $\bar{x}$ for $\bar{v} \in \partial f(\bar{x})$ if and only if there exist a number $\ve>0$, a neighborhood $U$ of $\bar{x}$, and an $f$-attentive $\ve$-localization of $\partial f$ around $(\bar{x},\bar{v})$ defined in \eqref{eq:fattentive} such that for all $x' \in U$ we have
\begin{equation}\label{eq:svarconv}
f(x') \ge f(x) + \la v, x' - x \ra + \frac{s}{2}\|x' - x\|^2\;\mbox{whenever }\;(x,v) \in \gph T_\ve.
\end{equation}
\end{Theorem}

It is said \cite[Definition~13.28]{Rockafellar98} that $f$ is {\em subdifferentially continuous} at $\bar{x}$ for $\bar{v}$ if for all $(x_k,v_k) \xrightarrow{\gph \partial f}(\bar{x},\bar{v})$ we have $f(x_k) \to f(\bar{x})$. The subdifferential continuity of $f$ at $\bar{x}$ for $\bar{v}$ allows us to replace $\gph T_\ve$ by an ordinary neighborhood of $(\bar{x},\bar{v})$ on $\gph \partial f$, i.e., the variationally $s$-convexity of $f$ at $\bar{x}$ for $\bar{v}$ holds if there exists $\varepsilon > 0$ such that for all $(x,v) \in \gph \partial f \cap B((\bar{x},\bar{v}),\varepsilon)$ we have
$$
f(x') \ge f(x) + \la v, x' - x \ra + \frac{s}{2}\|x' - x\|^2\;\mbox{ whenever }\;x'\in B(\bar{x},\varepsilon).
$$
In the next new result, important in what follows, we use the notation
$$
\gph T_{\varepsilon}(x,v): =\big\{ (\tilde{x},\tilde{v}) \in \gph \partial f\;\big|\;\tilde{x} \in B(x,\varepsilon),\;\tilde{v} \in B(v,\varepsilon),\;f(\tilde{x}) < f(x) + \varepsilon\big\}.
$$

\begin{Proposition}\label{prop:proxlan} Let $f: \R^n \to \overline{\R}$ be a proper l.s.c.\ function, and let $0 < \varepsilon_2 < \varepsilon_1$. Then:
\begin{itemize} 
\item[\rm \textbf{(i)}]  For all $(x,v) \in \gph T_{\varepsilon_2}(\bar{x},\bar{v})$, we have $\gph T_{\varepsilon_1 - \varepsilon_2}(x,v) \subset \gph T_{\varepsilon_1}(\bar{x},\bar{v})$.

\item[\rm \textbf{(ii)}] Fix any $s \in \R$ and assume that $f$ is variationally $s$-convex at $\bar{x}$ for $\bar{v}$. Let $\varepsilon_1$ be the corresponding radius, i.e., \eqref{eq:svarconv} holds for all $(x,v) \in \gph T_{\varepsilon_1}(\bar{x},\bar{v})$ and $x' \in B(\bar{x},\varepsilon_1)$. Then whenever $(x,v) \in \gph T_{\varepsilon_2}(\bar{x},\bar{v})$, $f$ is variationally $s$-convex at $x$ for $v$ with the radius $\varepsilon_1 - \varepsilon_2 > 0$.
\end{itemize}
\end{Proposition}
\begin{proof} \textbf{(i)} Fix $(x,v) \in \gph T_{\varepsilon_2}(\bar{x},\bar{v})$ and get for any $(\tilde{x},\tilde{v}) \in \gph T_{\varepsilon_1 - \varepsilon_2}(x,v)$ that $\tilde{x} \in B(x,\varepsilon_1 - \varepsilon_2)$, $\tilde{v} \in B(x,\varepsilon_1 - \varepsilon_2)$, which obviously yields $\tilde{x} \in B(\bar{x},\varepsilon_1)$ and $\tilde{v} \in B(\bar{v},\varepsilon_1)$. Moreover, we have
$$
f(\tilde{x}) < f(x) + \varepsilon_1 - \varepsilon_2 < [f(\bar{x}) + \varepsilon_2] + \varepsilon_1 - \varepsilon_2 = f(\bar{x}) + \varepsilon_1,
$$
and hence $\gph T_{\varepsilon_1 - \varepsilon_2}(x,v) \subset \gph T_{\varepsilon_1}(\bar{x},\bar{v})$. \medskip

\textbf{(ii)} Assume that $f: \R^n \to \overline{\R}$ is variationally $s$-convex at $\bar{x}$ for $\bar{v}$ with corresponding radius $\varepsilon_1$. Pick $(x,v) \in \gph T_{\varepsilon_1}(\bar{x},\bar{v})$ and get by \textbf{(i)} that $\gph T_{\varepsilon_1 - \varepsilon_2}(x,v) \subset \gph T_{\varepsilon_1}(\bar{x},\bar{v})$. Take then any $(\tilde{x},\tilde{v}) \in \gph T_{\varepsilon_1 - \varepsilon_2}(x,v) \subset \gph T_{\varepsilon_1}(\bar{x},\bar{v})$ and deduce from the variationally $s$-convexity of $f$ at $\bar{x}$ for $\bar{v}$ with radius $\varepsilon_1 > 0$ the fulfillment of the inequality
$$
f(x') \ge f(\tilde{x}) + \la \tilde{v}, x' - \tilde{x} \ra + \frac{s}{2}\|x' - \tilde{x}\|^2\;\mbox{ whenever }\;x'\in B(\bar{x},\varepsilon_1)
$$
Since we obviously have $B(x,\varepsilon_1 - \varepsilon_2) \subset B(\bar{x},\varepsilon_1)$, it follows that
$$
f(x') \ge f(\tilde{x}) + \la \tilde{v}, x' - \tilde{x} \ra + \frac{s}{2}\|x' - \tilde{x}\|^2\l\mbox{ for all }\;x'\in B(x,\varepsilon_1 - \varepsilon_2).
$$
Remembering that $(\tilde{x},\tilde{v})$ was taken arbitrarily in $\gph T_{\varepsilon_1 - \varepsilon_2}(x,v)$,
we arrive at the variational $s$-convexity of $f$ at $x$ for $v$ with the corresponding radius $\varepsilon_1 - \varepsilon_2$.
\end{proof}

The remaining part of this section addresses relationships between quadratic growth and second-order subderivatives with applications to variational $s$-convexity of extended-real-valued functions.

\begin{Theorem}\label{prop:strongchar} Let $f: \R^n \to \overline{\R}$ be an extended-real-valued l.s.c.\ function with $\bar{x} \in \dom f$, and let $\bar{v} \in \R^n$. Consider the following assertions:
 \begin{itemize}
\item[\rm\textbf{(i)}] There exists a neighborhood $U$ of $\bar{x}$ on which the second-order growth condition is satisfied
\begin{equation}\label{uni2}
f(x) \ge f(\bar{x}) + \la \bar{v}, x - \bar{x} \ra + \frac{\kappa}{2}\|x - \bar{x}\|^2,\quad x\in U.
\end{equation}

\item[\rm\textbf{(ii)}] We have the second-order subderivative estimate
\begin{equation}\label{tiltd22}
d^2 f(\bar{x}|\bar{v})(w) \ge \mu\|w\|^2\;\mbox{ for all }\;w\in\R^n.
\end{equation}
\end{itemize} 
Then implication {\rm \textbf{(i)} $\Longrightarrow$ \textbf{(ii)}} holds with $\mu = \kappa$, while the reverse one {\rm \textbf{(ii)} $\Longrightarrow$ \textbf{(i)}} holds with $\mu > \kappa$.
\end{Theorem}
\begin{proof} \textbf{} Assuming \textbf{(i)}, take any $(\bar{x},\bar{v}) \in \R^n \times \R^n$, $\bar{x} \in \dom f$, and a neighborhood $U$ of $\bar{x}$ on which \eqref{uni2} is satisfied. We have $x = \bar{x} + t w' \in U$ for any $w \in \R^n$, any $t > 0$ sufficiently small, and any $w' \in \R^n$ sufficiently close to $w$. This yields the estimate
$$
\frac{f(\bar{x} + t w') - f(\bar{x}) - \la \bar{v}, t w'\ra}{\frac{1}{2}t^2} \ge \kappa\|w'\|^2.
$$
Passing to the limit on both sides above brings us to
$$
d^2 f(\bar{x}|\bar{v})(w) = \liminf\limits_{\substack{t \downarrow 0\\w' \to w}}  \frac{f(\bar{x} + t w') - f(\bar{x}) - \la v, t w'\ra}{\frac{1}{2}t^2}\ge \liminf\limits_{\substack{t \downarrow 0\\w' \to w}} \kappa\|w'\|^2 = \kappa\|w\|^2,
$$
which thus justifies \eqref{tiltd22} with $\mu = \kappa$. 
\medskip

To verify the opposite implication, suppose that \textbf{(ii)} holds and that $\mu > \kappa$. Then 
\begin{equation}\label{unit u}
d^2 f(\bar{x}|\bar{v})(u) \ge \mu\;\mbox{ for all }\;u\in\R^n\;\mbox{ with }\;\|u\| = 1.
\end{equation}
Pick any $u$ from \eqref{unit u} and deduce from definition \eqref{eq:d2} of the second-order subderivative
that there exist $t_u > 0$ and a neighborhood $N_u$ of $u$ such that
$$
\Delta^2_t f(\bar{x}|\bar{v})(u')\ge\kappa\;\mbox{ for all }\;u'\in N_u\;\mbox{ and }\;t\in(0,t_u).
$$
Since the open sets $N_u$, $\|u\| = 1$ cover the compact set $\{u\in\R^n\;|\;\|u\| = 1\}$, we find finitely many $u_1,\ldots, u_m$ such that the latter set is covered by $\left( N_{u_i} \right)_{i = 1,\ldots,m}$. Define 
$$
t_0 := \min\big\{ t_{u_1},\ldots,t_{u_m}\big\}
$$
and observe that $\Delta^2_t f(\bar{x}|\bar{v})(u) \ge \kappa$ whenever $\|u\| = 1$ and $t \in (0,t_0)$. We have furthermore that $\Delta_t^2 f(\bar{x}|\bar{v})(0) = 0$ and that
$$
\Delta_{t}^2 f(\bar{x}|\bar{v})(w) = \Delta_{t}^2 f(\bar{x}|\bar{v})\Big( \|w\| \frac{w}{\|w\|}\Big) = \|w\|^2 \Delta_{t \|w\|}^2  f(\bar{x}|\bar{v})\Big( \frac{w}{\|w\|}\Big) \ge \|w\|^2\kappa
$$
for all $w \in \R^n \setminus \{0\}$ and all $t > 0$ with $\|t w\| < t_0$. Therefore, 
$$
f(\bar{x} + t w) \ge f(\bar{x}) + \la \bar{v}, t w \ra +  \frac{\kappa}{2}\|t w\|^2\;\mbox{ whenever }\;w \in \R^n,\;t > 0\;\mbox{ with }\;\|t w\| < t_0.
$$
Denoting $U := B(\bar{x},t_0)$, we arrive at \eqref{uni2} and thus complete the proof of the theorem.
\end{proof}

The following example shows that implication {\bf(ii)$\Longrightarrow$(i)} of Theorem~\ref{prop:strongchar} fails if $\mu=\kappa$.

\begin{Example}\label{exa1} {\rm Letting $f(x): = x^3 + \dfrac{\kappa}{2}x^2$ on $\R$ with $\bar{x} = 0$ and $\bar{v} = 0$, we see that $f''(x) = 6x + \kappa$ and
$$
d^2 f(0|0)(w) = f''(0).\;w^2 = \kappa w^2.
$$
Thus \eqref{tiltd22} holds. On the other hand,  \eqref{uni2} is equivalent to the existence of a neighborhood $U$ of $0$ such that $x^3 \ge 0$ for all $x \in U$, which cannot happen.}
\end{Example}

Note that Theorem~\ref{prop:strongchar} is an extension of \cite[Proposition~13.24(c)]{Rockafellar98} giving us a characterization of strong local minimizers of $f$ when $\bar{v} = 0$ via second-order subderivatives. Contrary to the latter, we {\em do not require} that $\bar{v} \in \partial f(\bar{x})$ for this equivalence to hold, and our result explicitly establish a relationship between {\em moduli} of the positive-definiteness of $d^2 f(\bar{x}|\bar{v})$ of  the strong local minimizer in question.\vspace*{0.03in}

The next assessment establishes an important quadratic estimate of second-order subderivatives for variationally $s$-convex extended-real-valued functions. 

\begin{Corollary}\label{cor:d2prox}  Given $s\in\R$, let $f: \R^n \to \overline{\R}$ be variationally $s$-convex at $\bar{x} \in \dom f$ for $\bar{v} \in \partial f(\bar{x})$. Then there exists an $f$-attentive $\ve$-localization $T_\ve$ of $\partial f$ around $(\bar{x},\bar{v})$ with some $\ve>0$ such that 
\begin{equation*}
d^2 f(x|v)(w) \ge s\|w\|^2\;\mbox{ for any }\;(x,v) \in \gph T_\ve\;\mbox{ and }\;w \in \R^n.
\end{equation*}
\end{Corollary}

\begin{proof} 
Assume that $f$ is variationally $s$-convex at $\bar{x}$ for $\bar{v}$ for some $s \in \R$. It follows from Theorem~\ref{theo:svarchar} that there exist an $f$-attentive $\ve$localization $T_\ve$ of $\partial f$ around $(\bar{x},\bar{v})$ and a neighborhood $U$ of $\bar{x}$ such that 
\begin{equation}\label{eq:growcondxv}
f(x') \ge f(x) + \la v, x' - x \ra + \frac{s}{2}\|x' - x\|^2\;\mbox{ whenever }\;(x,v) \in \gph T_\ve\;\mbox{ and }\; x' \in U.
\end{equation}
For any $(x,v) \in \gph T_\ve$, the set $U$ is also a neighborhood of $x$, and so \eqref{eq:growcondxv} is a second-order growth condition around $x$. Applying Theorem~\ref{prop:strongchar} for such $(x,v)$ tells us that
$$
d^2 f(x|v)(w) \ge s\|w\|^2\;\mbox{ for all }\;w \in \R^n,
$$
which therefore completes the proof of the corollary.
\end{proof}

\section{Properties of Generalized Twice Differentiable Functions}\label{sec:gtd}

In this section, we study the notion of \textit{generalized twice differentiability} of extended-real-valued function, the terms that has been recently coined by Rockafellar in \cite{roc}. This concept acts as a bridge between merely twice-epi differentiability and classical twice differentiability, requiring the good behavior of second-order subderivatives while not restricting the function too much. The history of this notion can be traced back to the paper \cite{genhes}, where a characterization for global minimizers was established in Theorem~3.8 therein. Following \cite[Definition~3.7]{genhes}, we have the following extension of ordinary quadratic forms. For any matrix $A \in \R^{n \times n}$, consider the quadratic function
$$
q_A(x) :=\la x, Ax \ra,\quad x\in\R^n.
$$
\begin{Definition}\label{def:quad} {\rm A function $q: \R^n \to (-\infty , \infty ]$ is called a {\em generalized quadratic form} if it is expressible as $q = \dfrac{1}{2}q_A + \delta_L$, where $L$ is a linear subspace of $\R^n$, and where $A \in \R^{n \times n}$ is a symmetric matrix.} 
\end{Definition}
It follows from Definition~\ref{def:quad} that $q(0) = 0$  and $\partial q$ is a \textit{generalized linear mapping}, i.e., $\gph \partial q$ is a subspace of $\R^n \times \R^n$. Generalized quadratic forms allow us to formulate the aforementioned notion from \cite{roc}.

\begin{Definition}\rm \label{defi:quaddiff}
A function $f: \R^n \to \overline{\R}$ is {\em generalized twice differentiable} at $\bar{x}$ for $\bar{v} \in \partial f(\bar{x})$ if it is twice epi-differentiable at $\bar{x}$ for $\bar{v}$ and its second subderivative $d^2 f(\bar{x}|\bar{v})$ is a generalized quadratic form.
\end{Definition}

It is easy to see that the classical twice differentiability of $f$ at $\ox$ yields the generalized one with
$$
d^2 f(\bar{x}|\nabla f(\bar{x}))(w) = \la w, \nabla^2 f(\bar{x}) w \ra\;\mbox{ for all }\;w\in\R^n
$$
The next theorem shows that generalized  twice differentiability may hold for (first-order) nonsmooth functions and fully clarifies this issue for any norm function on $\R^n$ with calculating its second subderivative. 

\begin{Theorem}\label{prop:normgtd} Let $f = \|\cdot\|$ be a norm $($not necessary Euclidean$)$ function on $\R^n$. Then we have:
\begin{itemize}
\item[\rm \textbf{(i)}] Whenever $v \in \partial f(0)$, $f$ is twice epi-differentiable at $0$ for $v$ and the corresponding second-order subderivative is calculated by
\begin{equation}\label{d2norm}
d^2 f(0|v) = \delta_{K(0,v)}\;\mbox{ as }\;v \in \partial f(0)
\end{equation}
via the critical cone $K(0,v)$ of $f$ at $0$ for $v$ taken from \eqref{eq:critcone}.

\item[\rm \textbf{(ii)}] $f$ is generalized twice differentiable at $0$ for $v$ if and only if $v \in \mathrm{int}\,\partial f(0)$.
\end{itemize}
\end{Theorem}
\begin{proof} It follows from \eqref{1sub} the first subderivative expressions
$$
df(0)(w) = \liminf\limits_{\substack{t \downarrow 0\\ w' \to w}} \frac{f(0 + t w') - f(0)}{t} = \liminf\limits_{\substack{t \downarrow 0\\ w' \to w}} \frac{t\|w'\|}{t} = \liminf\limits_{\substack{t \downarrow 0\\ w' \to w}} \|w'\| = \|w\|,\quad w\in\R^n,
$$  
which yield the critical cone representation
\begin{equation}\label{eq:critcone2}
K(0,v)=\big\{ w \in \R^n\;\big|\;df(0)(w) = \la v, w \ra\big\}=\big\{ w \in \R^n\;\big|\;\|w\| = \la v, w \ra \big \}\;\mbox{ for any  }\;v \in \partial f(0 ).
\end{equation}
To verify \textbf{(i)}, recall that the \textit{dual norm} of $\|\cdot\|$, denoted by $\|\cdot\|_*$, is defined by
$$
\|z\|_* := \sup\big\{ \la z, x \ra\;\big|\;\|x\| = 1\big\}. 
$$
We know that $\partial f(0) = \big\{ v \in \R^n\;\big|\; \|v\|_* \le 1\big\}$, which implies by the Cauchy-Schwarz inequality that
\begin{equation}\label{eq:CSnorm}
\la v, w \ra \le \|v\|_* \|w\| \le \|w\|\;\mbox{ for all }\;v \in \partial f(0),\quad w\in\R^n.
\end{equation}
Therefore, whenever $w \in \R^n$ and $t > 0$, we get
$$
\Delta_t^2 f(0|v)(w) = \dfrac{f(0 + tw) - f(0) - t\la v, w \ra}{\frac{1}{2}t^2} = \dfrac{\|w\| - \la v,w\ra}{\frac{1}{2}t}.
$$
It follows from \eqref{eq:CSnorm} that the family of functions $\{\Delta_t^2 f(0|v)\}_{t \in (0,\infty)}$ is 
nondecreasing as $t \downarrow 0$, which yields its epigraphical convergence by \cite[Proposition~7.4(d)]{Rockafellar98} and hence justifies the twice epi-differentiability of $f$ at $0$ for any $v\in\partial f(0)$. To prove {\bf(i)}, it remains verifying \eqref{d2norm}. Indeed, it follows from $\Delta_t^2 f(0|v) \ge 0$ as $t>0$ that $d^2 f(0|v) \ge 0$. Pick  any $w \in K(0,v)$ and deduce from Proposition~\ref{prop:epichar} that
$$
d^2 f(0|v)(w) \le \dfrac{\|w\| - \la v, w \ra}{\frac{1}{2}t} = 0,
$$
which tells us that $d^2 f(0|v)(w) = 0$ by \eqref{eq:CSnorm}. In the reverse direction, taking $w \notin K(0,v)$ gives us $\|w\| - \la v, w\ra > 0$, which amounts to saying that
$$
d^2 f(0|v)(w) = \liminf\limits_{\substack{t \downarrow 0 \\ w' \to w}} \Delta_t^2 f(0|v)(w') = \liminf\limits_{\substack{t \downarrow 0 \\ w' \to w}} \dfrac{\|w'\| - \la v,w'\ra}{\frac{1}{2}t} =\infty.
$$
This brings us to \eqref{d2norm} and completes the proof of {\bf(i)}.\vspace*{0.05in}

To proceed with the proof of \textbf{(ii)}, we obviously have $\mathrm{int}\,\partial f(0) = \{ v\in \R^n \mid \|v\|_* < 1 \}$. Pick any $v \in \R^n$ with $\|v\|_* < 1$ and fix $w \in K(0,v)$. It follows from  \eqref{eq:CSnorm} that
$$
\|w\| = \la v, w \ra \le \|v\|_*\|w\|,
$$
which yields $(1 - \|v\|_*).\|w\| \le 0$. Since $1 - \|v\|_* > 0$, the latter tells us that $w = 0$. Therefore, for all $v \in \mathrm{int}\,\partial f(0)$, we get $K(0,v) = \{0\}$, which leads us to $d^2 f(0|v) = \delta_{\{0\}}$. This readily justifies the generalized twice differentiability of $f$ is at $0$ for $v$ for any $v \in \mathrm{int}\,\partial f(0)$.\vspace*{0.03in} 

Now we show that for any $v\in\R^n$ with $\|v\|_* = 1$, the cone $K(0,v)$ cannot be a linear subspace and hence the generalized twice differentiability of $f = \|\cdot\|$ at $0$ for $v$ fails due to \eqref{d2norm}. Indeed, take any $v \in \R^n$ with $\|v\|_* = 1$ and deduce from the dual norm definition the existence of $\bar{x}$ such that $\|\bar{x}\| = 1$ and $\|v\|_* = \la v, \bar{x} \ra$. Since $\|v\|_* = 1$, it follows that
$$
\|\bar{x}\| = 1 = \la v,\bar{x} \ra.
$$
By \eqref{eq:critcone2}, the latter amounts to $\bar{x} \in K(0,v)$. Assuming by the contrary that $K(0,v)$ is a linear subspace ensures that $-\bar{x} \in K(0,v)$, and hence
$$
-\|-\bar{x}\| = \la v, \bar{x} \ra = \|\bar{x}\|.
$$
This yields $\|\bar{x}\| = 0$, which is a contradiction. Therefore, for each $v \in \R^n$ with $\|v\|_* = 1$, the cone $K(0,v)$ cannot be a linear subspace, and thus $f$ is not generalized twice differentiable at $0$ for $v$.
\end{proof}

On the other hand, the next theorem shows that {\em generalized} twice differentiability for {\em $\mathcal{C}^{1,1}$ functions} reduces to their {\em classical} twice differentiability at the point in question. The following technical result taken from \cite[Lemma A.11]{newton} is useful in what follows.

\begin{Lemma}\label{lem:c11ineq} Let $f: \R^n \to \R$ be differentiable on the given interval $[x_1,x_2]\subset\R^n$, and let the gradient $\nabla f$ be Lipschitz continuous on that interval with constant $L > 0$. Then we have
$$
|f(x_2) - f(x_1) - \la \nabla f(x_1), x_2 - x_1 \ra | \le \frac{L}{2} \|x_2 - x_1 \|^2. 
$$ 
\end{Lemma}

Here is the aforementioned theorem whose both parts are important for the subsequent applications. 

\begin{Theorem}\label{prop:d2finite} Let $f: \R^n \to \overline{\R}$ be differentiable on a neighborhood $U$ of $\bar{x}$ having the Lipschitz continuous gradient $\nabla f$ on $U$ with modulus $L > 0$. The following assertions hold:
\begin{itemize}
\item[\rm\textbf{(i)}] For all $x \in U$, we have the estimate
\begin{equation}\label{eq:d2finite}
\big|d^2 f\big(x|\nabla f(x)\big)(w)\big| \le L \|w\|^2\;\mbox{ whenever }\; w \in \R^n,
\end{equation}
which ensures, in particular, that $d^2 f(x|\nabla f(x))$ is finite on $\R^n$.

\item[\rm\textbf{(ii)}] For any point $x \in U$, the generalized twice differentiability of $f$ at $x$ for $\nabla f(x)$ is equivalent to the twice differentiability of $f$ at $x$ in the classical sense.
\end{itemize}
\end{Theorem}
\begin{proof} To verify \textbf{(i)}, fix any vector $w \in \R^n$ and apply Lemma~\ref{lem:c11ineq}, which tells us that for all $w'$ sufficiently close to $w$ and all $t>0$ sufficiently small with $x + tw' \in U$, we get
$$
|f(x + t w') - f(x) - \la \nabla f(x), t w' \ra\big| \le \frac{L}{2} t^2 \|w'\|^2. 
$$
This yields, for all such $w'$ and $t$, the estimate
$$
\big|\Delta_t^2 f\big(x|\nabla f(x)\big)(w')| \le L\|w'\|^2.
$$
Passing to the limits inferior when $t \downarrow 0$, $w' \to w$ and using \eqref{eq:d2} bring us to
$$
|d^2 f(x|\nabla f(x))(w)| \le L \|w\|^2,
$$
which clearly implies that $d^2 f(x|\nabla f(x))$ is finite.\medskip

To proceed with the proof of \textbf{(ii)}, we only need to show that if $f$ is generalized twice differentiable at some $x\in U$ for $\nabla f(x)$, then $f$ is twice differentiable at $x$ in the classical sense. Assuming that $f$ is generalized twice differentiable at $x$ for $\nabla f(x)$, we find a symmetric matrix $A$ and a linear subspace $L$ such that
\begin{equation}\label{eq:d2c11gtd}
d^2 f(x|\nabla f(x)) = q_A + \delta_L.
\end{equation}
Since $f$ is of class $\mathcal{C}^{1,1}$ around $x$, it is prox-regular everywhere near $x$ by \cite[Proposition~13.34]{Rockafellar98}. It follows from assertion \textbf{(i)} that $d^2f(x|\nabla f(x))$ is finite everywhere, i.e., the linear subspace $L$ in \eqref{eq:d2c11gtd} is the entire $\R^n$. Applying \cite[Theorem~6.7]{proxreg} tells us that $f$ has a quadratic expansion at $x$ meaning that
$$
\begin{aligned}
f(y) &= f(x) + \la \nabla f(x), y - x \ra + \frac{1}{2}d^2 f(x|\nabla f(x))(y - x) + o(|y - x|^2)\\
&= f(x) + \la \nabla f(x), y - x \ra + \frac{1}{2}\la y - x,A(y - x) \ra + o(|y - x|^2)\;\mbox{ for all }\;y\;\mbox{ near }\;x.
\end{aligned}
$$
Combining finally \cite[Corollary~13.42]{Rockafellar98} with the differentiability $f$ on a neighborhood of $x$, we conclude that $f$ is twice differentiable at $x$ in the classical sense and thus complete the proof of the theorem.
\end{proof}

Now we present two examples showing that the assumptions of Theorem~\ref{prop:d2finite} are essential for the conclusion.
\begin{Example}\label{rmk:C1gtd}
The conclusion of Theorem~{\rm\ref{prop:d2finite}(ii)} fails if $f$ is $\mathcal{C}^1$-smooth but not of class  $\mathcal{C}^{1,1}$ around $\ox$.
\end{Example}
\begin{proof} Consider the function $f: \R \to \R$ defined by $f(x) := |x|^{3/2},\;x \in \R$. Then $f$ is $\mathcal{C}^1$-smooth around $\bar{x} = 0$ but is not of class $\mathcal{C}^{1,1}$ around that point. This function is obviously not twice differentiable at $\bar{x} = 0$ in the classical sense. Let us check that $f$ is generalized twice differentiable at $\bar{x} = 0$ for $\bar{v}:=f(\bar x) = 0$ . To calculate $d^2 f(0|0)$, we first get that
$$
\Delta_t^2 f(0|0)(w) = \frac{f(0 + tw) - f(0) - t.0.w}{\frac{1}{2}t^2} = \frac{|tw|^{3/2}}{\frac{1}{2}t^2} = \frac{|w|^{3/2}}{\frac{1}{2}\sqrt{t}}\;\mbox{ for all }\;w\in\R^n\;\mbox{ and }\;t>0.
$$
It follows from Proposition~\ref{prop:epichar} that
\begin{equation}\label{eq:gtdc2d1}
\Big[\eliminf\limits_{t \downarrow 0} \Delta_t^2 f(0|0)\Big](0) = \Big[\elimsup\limits_{t \downarrow 0} \Delta_t^2 f(0|0)\Big](0) = 0, 
\end{equation}
which readily tells us by  \eqref{eq:d2} that
$$
d^2 f(0|0)(1) = d^2 f(0|0)(-1) =\infty.
$$
The positive homogeneity of degree $2$ of $d^2 f(0|0)$ ensures that $d^2 f(0|0) = \delta_{\{0\}}$, which is a generalized quadratic form on $\R$. Furthermore, observe the relationships 
\begin{equation}\label{eq:gtdc2d2}
\Big[\eliminf\limits_{t \downarrow 0} \Delta_t^2 f(0|0)\Big](w) = \Big[\elimsup\limits_{t \downarrow 0} \Delta_t^2 f(0|0)\Big](w) = \infty,\quad  w \in\R,
\end{equation}
which being combined with \eqref{eq:gtdc2d1} and \eqref{eq:gtdc2d2} establish twice epi-differentiability of $f$ at $0$ for $0$. 
\end{proof}

\begin{Example}\label{rmk:gtd} The conclusion of Theorem~{\rm\ref{prop:d2finite}(ii)} fails if the generalized twice differentiability therein is replaced by merely twice epi-differentiability.
\end{Example}
    
\begin{proof} Consider the function $f: \R \to \R$ defined by $f(x) := x^2 \mathrm{sgn}(x)$ with $\nabla f(x)=|x|$ for all $x\in\R$, which is of class $\mathcal{C}^{1,1}$ around $\bar{x} = 0$. To check that $f$ is twice epi-differentiable at $0$ for $0$, we have
$$
\Delta_t^2 f(0|0)(w) = \frac{t^2w^2 \sgn(tw) - 0 - t\cdot 0\cdot w}{\frac{1}{2}t^2} = 2w^2 \sgn(w)\;\mbox{ for all }\;t>0\;\mbox{ and }\;w\in\R^n.
$$
Observing that the function $\Delta_t^2 f(0|0)(w)$ is continuous everywhere yields $\Delta_t^2 f(0|0)(w') \to \Delta_t^2 f(0|0)(w)$ as $w' \to w$. Apply Proposition~\ref{prop:epichar} to get the expressions
$$
\Big[\eliminf\limits_{t \downarrow 0} \Delta_t^2 f(0|0)\Big](w) = \Big[\elimsup\limits_{t \downarrow 0} \Delta_t^2 f(0|0)\Big](w) = 2 w^2 \sgn(w),\quad w\in\R.
$$
Thus $\Delta_t^2 f(0|0)(w)$ epigraphically converges as $t\dn 0$ to the function $d^2 f(0|0)(w) = 2 w^2 \sgn (w),\;w \in \R$, which justifies the twice epi-differentiability of $f$ at $0$ for $0$. On the other hand, $d^2 f(0|0)$ is not a generalized quadratic form on $\R$ because any generalized quadratic form $q$ on $\R$ must either be $q(w) = aw^2$ for all $w \in \R$ with $a \in \R$, or $q = \delta_{\{0\}}$. Therefore, $f$ is not generalized twice differentiable at $0$ for $0$.
\end{proof}

The next proposition establishes the {\em preservation} of generalized twice differentiability for extended-real-valued functions under twice differentiable additions.

\begin{Proposition}\label{prop:sumC2gendif}  Let $g: \R^n \to \overline{\R}$ be l.s.c.\ with $\bar{v} \in \partial g(\bar{x})$, and let $f: \R^n \to \overline{\R}$ be strictly differentiable at $\bar{x} \in \R^n$ and twice differentiable at this point in the classical senses. The following assertions hold:
\begin{itemize}
\item[\rm\textbf{(i)}] For all $w\in\R^n$, we have the equality
$$
d^2 (f+g)(\bar{x}|\nabla f(\bar{x}) + \bar{v})(w) = \la w, \nabla^2 f(\bar{x})w \ra + d^2 g(\bar{x}|\bar{v})(w). 
$$
\item[\rm\textbf{(ii)}] If $g$ is generalized twice differentiable at $\bar{x}$ for $\bar{v}$, then the summation function $f + g$ is generalized twice differentiable at $\bar{x}$ for $\nabla f(\bar{x}) + \bar{v}$.
\end{itemize}
\end{Proposition}
\begin{proof} \textbf{(i)} Since $f$ is strictly differentiable at $\bar{x}$, it follows that $\nabla f(\bar{x}) + \bar{v} \in \partial (f + g)(\bar{x})$ by \cite[Proposition~1.107(ii)]{Mordukhovich06}. Since $f$ is twice differentiable at $\bar{x}$, we have
$$
\begin{aligned}
d^2 (f+g)(\bar{x}|\nabla f(\bar{x}) + \bar{v})(w)  &= \liminf\limits_{\substack{t \downarrow 0\\ w' \to w}} \frac{(f+g)(\bar{x} + tw') - (f+g)(\bar{x}) - t \la \nabla f(\bar{x}) + \bar{v}, w' \ra}{\frac{1}{2}t^2}\\
&= \liminf\limits_{\substack{t \downarrow 0\\ w' \to w}} \left[ \frac{f(\bar{x} + tw') - f(\bar{x}) - t \la \nabla f(\bar{x}), w' \ra}{\frac{1}{2}t^2} + \frac{g(\bar{x} + tw') - g(\bar{x}) - t \la \bar{v}, w' \ra}{\frac{1}{2}t^2} \right]\\
&= \lim\limits_{\substack{t \downarrow 0\\ w' \to w}}\frac{f(\bar{x} + tw') - f(\bar{x}) - t \la \nabla f(\bar{x}), w' \ra}{\frac{1}{2}t^2} + \liminf\limits_{\substack{t \downarrow 0\\ w' \to w}} \frac{g(\bar{x} + tw') - g(\bar{x}) - t \la \bar{v}, w' \ra}{\frac{1}{2}t^2}\\
&= \la w, \nabla^2 f(\bar{x})w \ra + d^2 g(\bar{x}|\bar{v})(w),\quad w\in\R^n.
\end{aligned} 
$$

\textbf{(ii)} By the generalized twice differentiability of $g$ at $\bar{x}$ for $\bar{v}$, there exist a symmetric matrix $A$ and a subspace $L$ of $\R^n$ such that $d^2g(\bar{x}|\bar{v}) = q_A + \delta_L$. It follows from \textbf{(i)} that
\begin{equation}\label{eq:propsumd2}
d^2(f+g)(\bar{x}|\nabla f(\bar{x}) + \bar{v})(w) = \la w, (\nabla^2 f(\bar{x}) + A)w \ra + \delta_L(w),\quad w\in\R^n.
\end{equation}
Since $g$ is  properly twice epi-differentiable at $\bar{x}$ for $\bar{v}$ and $f$ is twice differentiable at $\bar{x}$, we have that $\Delta_t^2 g(\bar{x}|\bar{v}) \xrightarrow{e} d^2 g(\bar{x}|\bar{v})$ and $\Delta_t^2 f(\bar{x}|\nabla f(\bar{x})) \xrightarrow{c} q_{\nabla^2 f(\bar{x})}$ as $t \downarrow 0$. Applying 
\cite[Theorem~7.46(b)]{Rockafellar98} together with \eqref{eq:propsumd2} gives us as $t\dn0$ that
$$
\Delta_t^2 (f + g)(\bar{x}|\nabla f(\bar{x}) + \bar{v}) =  \Delta_t^2 f(\bar{x}|\nabla f(\bar{x})) + \Delta_t^2 g(\bar{x}|\bar{v}) \xrightarrow{e} d^2 (f+g)(\bar{x}|\nabla f(\bar{x}) + \bar{v})
$$
thus verifying the twice epi-differentiability  of $f+g$ at $\bar{x}$ for $\nabla f(\bar{x}) + \bar{v}$. Therefore, $f + g$ is generalized twice differentiable at $\bar{x}$ for $\nabla f(\bar{x}) + \bar{v}$ as claimed.
\end{proof}

\section{Generalized Twice Differentiability via Moreau Envelopes}\label{gen-moreau}

The main result of this section establishes the {\em equivalence} between {\em generalized twice differentiability} for the broad class of (extended-real-valued)  prox-regular and prox-bounded functions and its {\em classical twice differentiability} counterpart for the associated {\em Moreau envelopes}. To achieve this goal, we present several auxiliary statements of their own interest. The first lemma is extracted from \cite[Proposition~13.37]{Rockafellar98}.

\begin{Lemma}\label{prop:prox-mor-pr} Let $f:\R^n \to \overline{\R}$ be prox-bounded on $\R^n$ and $r$-level prox-regular at $\bar{x}$ for $\bar{v} \in\partial f(\bar{x})$ with the corresponding radius $\varepsilon > 0$. Then for all $\lambda \in (0,1/r)$, there exists a neighborhood of $\bar{x} + \lambda\bar{v}$ on which:
\begin{itemize}
\item[\rm\textbf{(i)}] The proximal mappings $P_\lambda f$ is single-valued and Lipschitz continuous with $P_\lambda f(\bar{x} + \lambda\bar{v}) = \bar{x} + \lambda\bar{v}$.

\item[\rm\textbf{(ii)}] The Moreau envelope $e_\lambda f$ is of class $\mathcal{C}^{1,1}$ and such that
\begin{equation}\label{eq:gradenv}
\nabla e_\lambda f = \lambda^{-1}[I - P_\lambda f] = [\lambda I + T_\ve^{-1}]^{-1},
\end{equation}
where $T_\ve$ is the $f$-attentive $\varepsilon$-localization of $\partial f$ around $(\bar{x},\bar{v})$.
\end{itemize}
Moreover, the set-valued mapping $T_\ve$ in {\rm(ii)} can be chosen so that the set $U_{\lambda} := \rge (I + \lambda T_\varepsilon)$ serves for all $\lambda > 0$ sufficiently small as a neighborhood of $\bar{x} + \lambda\bar{v}$ on which these properties hold.
\end{Lemma}

The next technical observation is useful in what follows.

\begin{Lemma}\label{prop:shear} For any set-valued mapping $F: \R^n \rightrightarrows \R^n$ and any $\lambda \in \R$, we have
$$
(u,v) \in \gph [\lambda I + F^{-1}]^{-1} \Longleftrightarrow (u - \lambda v, v) \in \gph F,\;\mbox{ i.e., }\;
$$
$$
\gph F = A_{\lambda} [\gph [\lambda I + F^{-1}]^{-1}]\;\mbox{ for all }\;x, y \in \R^n,
$$
where $A_{\lambda} : \R^n \times \R^n \to \R^n \times \R^n$ is defined by 
$$
A_{\lambda}(x,y) := (x - \lambda y, y),
$$ 
which is an invertible linear mapping with $A_{\lambda}^{-1}(x,y) = (x + \lambda y, y)$.
\end{Lemma}
\begin{proof} Based on the definitions, we have the equivalences
$$
\begin{aligned}
(u,v) \in \gph [\lambda I + F^{-1}]^{-1} &\Longleftrightarrow (v,u) \in \gph [\lambda I + F^{-1}] \\
&\Longleftrightarrow (v,u - \lambda v) \in \gph F^{-1} \\
&\Longleftrightarrow (u - \lambda v,v) \in \gph F, 
\end{aligned} 
$$
which readily yield the other statements of the lemma.
\end{proof}

Yet another lemma provides a desired ingredient for subsequent limiting procedures. 

\begin{Lemma}\label{lem:zkvkxk} Let $f: \R^n \to \overline{\R}$ be prox-regular at $\bar{x}$ for $\bar{v} \in \partial f(\bar{x})$. Take $\lambda > 0$ to be so small that all the assertions in Lemma~{\rm\ref{prop:prox-mor-pr}} hold; in particular, $e_{\lambda}f$ is of class $\mathcal{C}^{1,1}$ around $\bar{z}:= \bar{x} + \lambda \bar{v}$. For any sequence $z_k \to \bar{z}$, define the vectors
\begin{equation}\label{eq:vkxk}
 v_k := \nabla  e_{\lambda} f(z_k), \quad x_k := z_k - \lambda v_k\;\mbox{ as }\;k\in\N,  
\end{equation}
Then we have the sequence of subgradients $v_k \in \partial f(x_k)$ such that
$$
(x_k,v_k) \to (\bar{x},\bar{v}), \quad f(x_k) \to f(\bar{x})\;\mbox{ as }\;k\to\infty.
$$
\end{Lemma}
\begin{proof} Assuming that $f: \R^n \to \overline{\R}$ is prox-regular at $\bar{x}$ for $\bar{v} \in \partial f(\bar{x})$ with the corresponding radius $\varepsilon_0 > 0$, we fix any $\varepsilon \in (0,\varepsilon_0)$ and deduce from  Lemma~\ref{prop:prox-mor-pr} that the representations in \eqref{eq:gradenv} hold on the neighborhood $U_{\lambda}^\varepsilon := \rge (I + \lambda T_{\varepsilon})$ of $\bar{z}$. Take a sequence $z_k \to \bar{z}$ and construct $v_k, x_k$ as in \eqref{eq:vkxk}. We get $z_k \in U_{\lambda}^\varepsilon$ for all $k$ sufficiently large and thus deduce from \eqref{eq:gradenv} that
$$
v_k = [\lambda I + T_\varepsilon^{-1}]^{-1}(z_k),
$$
which brings us to the equivalences
$$
(v_k,z_k) \in \gph [\lambda I + T_\varepsilon^{-1}] \Longleftrightarrow (v_k,x_k) \in \gph T_\varepsilon^{-1} \Longleftrightarrow (x_k,v_k) \in \gph T_\varepsilon.
$$
Therefore, it holds for all $k$ sufficiently large that 
$$
v_k \in \partial f(x_k), \quad x_k \in B(\bar{x},\varepsilon), \quad v_k \in B(\bar{v},\varepsilon), \quad f(x_k) < f(\bar{x}) + \varepsilon.
$$
Since $\varepsilon > 0$ was chosen arbitrarily small, we arrive at the claim conclusion.
\end{proof}

The following proposition, extending with the {\em modulus interplay} the corresponding statement of \cite[Exercise~13.45]{Rockafellar98}, provides the equivalence between twice epi-differentiability of a prox-regular function and the associated Moreau envelope. 

\begin{Proposition}\label{prop:twice-epi} Let $f: \R^n \to \overline{\R}$ be prox-bounded on $\R^n$ and is $r$-level prox-regular at $\bar{x}$ for $\bar{v} \in \partial f(\bar{x})$. Then for all $\lambda \in (0,1/r)$, the following assertions are equivalent:
\begin{itemize}
\item[\rm\textbf{(i)}] $f$ is twice epi-differentiable at $\bar{x}$ for $\bar{v}$.

\item[\rm\textbf{(ii)}] The Moreau envelope $e_{\lambda} f$ is twice epi-differentiable at $\bar{x} + \lambda \bar{v}$ for $\bar{v}$.
\end{itemize}
Moreover, we have under the assumptions above that
\begin{equation}\label{eq:envd^2}
e_{\lambda}\Big[ \frac{1}{2}d^2 f (\bar{x}|\bar{v})\Big] =  d^2\Big[ \frac{1}{2} e_{\lambda} f\Big]\Big(\bar{x} + \lambda\bar{v}\Big|\dfrac{1}{2}\bar{v}\Big).
\end{equation}
\end{Proposition}
\begin{proof}
Consider the shifted function
$$
g(x) := f(x) - \la \bar{v}, x - \bar{x} \ra,\quad x\in\R^n
$$
and observe by \cite[Exercise~13.35]{Rockafellar98} that $g$ is $r$-level prox-regular at $\bar{x}$ for $0$. Since $f$ is prox-bounded and $g$ differs from $f$ by a linear term, $g$ is also prox-bounded. Deduce now from by \cite[Lemma 2.2]{hare} that
\begin{equation}\label{eq:env_gf}
e_{\lambda} f(\bar{x} + \lambda \bar{v} + x) = e_{\lambda} g(\bar{x} + x) + \frac{\lambda}{2}\|\bar{v}\|^2 + \la \bar{v}, x \ra\;\mbox{ for all }\;\lambda > 0,\;x \in \R^n.
\end{equation}
It follows from the proof of Proposition~\ref{prop:sumC2gendif} that $g$ is twice epi-differentiable at $\bar{x}$ for $0$ with
\begin{equation}\label{eq:d^2gf}
d^2 g(\bar{x}|0) = d^2 f(\bar{x}|\bar{v}).
\end{equation}
Similarly we get by the usage of \eqref{eq:env_gf} that the twice epi-differentiability of $e_\lambda f$ at $\bar{x} + \lambda \bar{v}$ for $\bar{v}$ is equivalent to the twice epi-differentiability of $e_\lambda g$ at $\bar{x}$ for $0$ with the fulfillment of
\begin{equation}\label{eq:d^2gf2}
 d^2\Big[\frac{1}{2}e_{\lambda} f\Big]\Big(\bar{x} + \lambda\bar{v}\Big|\dfrac{1}{2}\bar{v}\Big) =  d^2\Big[\frac{1}{2}e_{\lambda} g\Big](\bar{x}|0).
\end{equation}
Applying \cite[Exercise~13.45]{Rockafellar98} to $g$, and combining this with \eqref{eq:d^2gf} and \eqref{eq:d^2gf2} establishes the equivalence between {\bf(i)} and {\bf(ii)}. The claimed equality \eqref{eq:envd^2} also follows from the above due to
$$
d^2\Big[\frac{1}{2}e_{\lambda} f\Big]\Big(\bar{x} + \lambda\bar{v} \left|\dfrac{1}{2}\bar{v} \right.\Big) =  d^2 \Big[\frac{1}{2}e_{\lambda} g\Big](\bar{x}|0) = e_{\lambda}\Big[\frac{1}{2}d^2 g(\bar{x}|0)\Big] = e_{\lambda} \Big[\frac{1}{2}d^2 f(\bar{x}|\bar{v})\Big],
$$
which therefore completes the proof of the proposition.
\end{proof}

If in addition the second-order subderivative $d^2 f(\bar{x}|\bar{v})$ is a {\em generalized quadratic form}, i.e., $f$ is generalized twice differentiable at $\bar{x}$ for $\bar{v}$, then we establish below much stronger conclusion telling us that the associated Moreau is not merely twice epi-differentiable but {\em twice differentiable} in the classical sense. To proceed, we need to derive three more auxiliary results as follows.

\begin{Lemma}\label{lem:sol} Let $L \ne \{0\}$ be a linear subspace of $\R^n$ with dimension $m\le n$, let $M \in \R^{n \times n}$ be a positive-definite matrix, and let $w \in \R^n$ be a fixed vector. Given an orthogonal basis $\{v_1,\ldots,v_m\}$ of $L$ and an $m\times n$ matrix $B$ with columns $v_i$, $i = 1,\ldots, m$, we claim that the matrix $B^TMB$ is invertible and $B(B^TMB)^{-1}B^T w$ is a unique solution to the system
\begin{equation}\label{eq:system}
\left\{ \begin{array}{l}
x \in L, \\
Mx - w \in L^\perp.  
\end{array} \right.
\end{equation}
\end{Lemma}
\begin{proof} To verify the invertibility of $B^TMB$, take any $u \in \R^m$ such that $B^TMBu = 0$ and get $(Bu)^TMBu = 0$. Since $M$ is positive-definite, it follows that $Bu = 0$, which yields $u = 0$ by the linear independence of the columns of $B$. To check further that $B(B^TMB)^{-1}B^T w$ is a solution to \eqref{eq:system}, observe that $B(B^TMB)^{-1}B^T w \in L$, which amounts to the fulfillment of the first inclusion in \eqref{eq:system}. For the second one in \eqref{eq:system}, pick any $u \in \R^m$ and get
$$
(Bu)^T(MB(B^TMB)^{-1}B^T w-w) = u^T(B^TMB(B^TMB)^{-1}B^Tw - B^Tw) = u^T(B^Tw - B^Tw) = 0,
$$
which shows that $B(B^TMB)^{-1}B^T w$ satisfies \eqref{eq:system}. To verify the uniqueness of solutions to \eqref{eq:system}, take two vectors $x_1$, $x_2$ for which $x_1 - x_2 \in L$ and $M(x_1 - x_2) \in L^\bot$. This yields $(x_1 - x_2)^T M(x_1 - x_2) = 0$ telling us that $x_1 - x_2 = 0$ by the positive-definiteness of $M$ and thus completing the proof. 
\end{proof}

The next lemma concerns relationships between proto-differentiability of mappings and their graphical derivatives under invertible linear operators. 

\begin{Lemma}\label{lem:gphDGF} Let $F$ and $G$  be set-valued mappings from $\R^n$ to $\R^m$ such that $\gph G = A(\gph F)$, where $A:\R^n\times \R^m\to\R^n \times\R^m$ is a linear invertible operator. Then given 
$(\bar{x},\bar{v}) \in \R^n \times \R^m$ and $(\bar{u},\bar{w}) := A(\bar{x},\bar{v})$, the 
proto-differentiability of $F$ at $\bar{x}$ for $\bar{v}$ is equivalent to the proto-differentiability of $G$ at $\bar{u}$ for $\bar{w}$. In that case, we have the graphical derivative relationship
\begin{equation}\label{proto-gph}
\gph D G(\bar{u}|\bar{w}) = A[\gph D F(\bar{x}|\bar{v})].
\end{equation}
\end{Lemma}
\begin{proof} For all positive numbers $t$, we obviously get
$$
\gph \Delta_t G (\bar{u}|\bar{w}) = \frac{\gph G - (\bar{u},\bar{w})}{t} = \frac{A(\gph F) - A(\bar{x},\bar{v})}{t}.
$$
which readily yields the equalities
$$
A^{-1}[\gph \Delta_t G (\bar{u}|\bar{v})]  = \frac{\gph F - (\bar{x},\bar{v})}{t} = \gph \Delta_t F(\bar{x}|\bar{v}),
$$
$$
\gph \Delta_t G(\bar{u}|\bar{w}) = A[\gph \Delta_t F(\bar{x}|\bar{v})].
$$
Assuming now that $G$ is proto-differentiable at $\bar{u}$ for $\bar{w}$ and taking the limits on both sides of the above equation as $t \downarrow 0$ bring us to the expression
\begin{equation}\label{eq:DG}
\gph DG(\bar{u}|\bar{w}) = \Lim\limits_{t \downarrow 0} A[\gph \Delta_t F(\bar{x}|\bar{v})].        
\end{equation}
Check that $A$ is norm-coercive. Indeed, the invertibility of $A$ ensures that $\|A^{-1}\| \ne 0$ and
$$
\|x\| = \|A^{-1}Ax\| \le \|A^{-1}\|\cdot\|A(x)\|\;\mbox{ for all }\;x \in \R^{n\times m}.
$$
Therefore, the convergence $\|x\| \to \infty$ obviously yields $\|A(x)\| \to \infty$. The norm-coercivity of $A^{-1}$ can be checked similarly. Further, rewrite the right-hand side of \eqref{eq:DG} in the form
\begin{equation*}
\gph DG(\bar{u}|\bar{w}) = A \Big[ \Lim\limits_{t \downarrow 0} \gph \Delta_t F(\bar{x}|\bar{v})\Big]        
\end{equation*}
and conclude that the outer limit of $\gph \Delta_t F(\bar{x}|\bar{v})$ exists as $t \downarrow 0$, which means
the proto-differentiability of $F$ at $\bar{x}$ for $\bar{v}$. Moreover, this gives us \eqref{proto-gph}. Assuming on the other hand that $F$ is proto-differentiable at $\bar{x}$ for $\bar{v}$, we write $\gph F = A^{-1}(\gph G)$ and deduce from the above that $G$ is proto-differentiable at $\bar u$ for $\bar w$ together with the fulfillment of the graphical derivative formula \eqref{proto-gph}.
\end{proof}

The final lemma concerns eigenvalue lower-boundedness of self-adjoint linear operators on subspaces.

\begin{Lemma}\label{lem:posdef} Let $L$ be a subspace of $\R^n$, let  $A: \R^n \to \R^n$ be a self-adjoint linear operator, and let $\sigma \in \R$. Impose on $L$ the  eigenvalue lower-boundedness
$$
\la w, Aw \ra \ge \sigma \|w\|^2\;\mbox{ for all }\;w \in L.
$$
Then there exists a self-adjoint linear operator $B: \R^n \to \R^n$ such that the eigenvalue lower-boundedness property of $A$ on $L$ ix extended to to the same property for $B$ on the entire space, i.e.. 
$$\la w,Bw \ra = \la w, Aw \ra\;\mbox{ for all }\;w\in L,
$$
$$\la w, Bw \ra \ge \sigma \|w\|^2 \;\mbox{ for all }\;w\in\R^n.
$$
\end{Lemma}
\begin{proof} Recalling that both orthogonal projector operators $P_L$ and $P_{L^\perp}$ are self-adjoint, define the self-adjoint linear operator $B$  by 
$$
B := P_L AP_L + \sigma P_{L^\perp}. 
$$
For any $w \in \R^n$, we get the equalities
$$\begin{aligned}
\la w, B w \ra &= \la w, P_L AP_L w \ra + \sigma\la w,  P_{L^\perp}w \ra \\
 &= \la P_Lw,  AP_L w \ra + \sigma\la w,  P_{L^\perp}P_{L^\perp}w \ra \\
 &= \la P_Lw, AP_Lw \ra + \sigma \| P_{L^\perp}w\|^2.
\end{aligned} $$
Since $P_{L^\perp}(w) = 0$ and $P_L(w) = w$ when $w \in L$, kit follows that $\la w, B w \ra = \la w, Aw \ra $. Moreover, by $P_Lw \in L$ and  \eqref{eq:pythagoras}, the latter implies that 
$$
\la w, Bw \ra = \la P_Lw, AP_Lw \ra + \sigma \| P_{L^\perp}w\|^2 \ge \sigma \|P_L w\|^2 + \sigma \| P_{L^\perp}w\|^2 = \sigma \|w\|^2, 
$$
which therefore verify our claims.
\end{proof}

Now we are ready to  establish the main result of this section.

\begin{Theorem}\label{prop:gtdtwicediff} Let $f: \R^n \to \overline{\R}$ be  prox-bounded on $\R^n$ and  $r$-level prox-regular at $\bar{x}$ for $\bar{v} \in \partial f(\bar{x})$. Then the following are equivalent for all $\lambda \in (0,1/r)$:
\begin{itemize}
\item[{\rm \textbf{(i)}}] $f$ is generalized twice differentiable at $\bar{x}$ for $\bar{v}$. 
\item[{\rm \textbf{(ii)}}] $e_{\lambda} f$ is twice differentiable at $\bar{x} + \lambda \bar{v}$.
\end{itemize}
\end{Theorem}
\begin{proof} Fix $\lambda \in (0,1/r)$ and first verify implication {\bf(i)$\Longrightarrow$(ii)}. By the generalized twice differentiability of $f$ at $\bar{x}$ for $\bar{v}$, there exist a symmetric matrix $A \in \R^{n \times n}$ and a subspace $L \subset \R^n$ such that
\begin{equation}\label{eq:d2fgtd}
 d^2 f(\bar{x}|\bar{v}) = q_A + \delta_L.
\end{equation}
Since $f$ is $r$-level prox-regular at $\bar{x}$ for $\bar{v} \in \partial f(\bar{x})$, it follows from 
Corollary~\ref{cor:d2prox} combined with \eqref{eq:d2fgtd} that
\begin{equation}\label{eq:assessmentA}
\la x, Ax \ra \ge -r\|x\|^2\;\mbox{ for all }\;x\in L.
\end{equation}
Using \eqref{eq:assessmentA}, apply Lemma~\ref{lem:posdef} and find a matrix $B$ such that
\begin{equation}\label{eq:assessmentA2}
\quad \la x, Bx \ra = \la x, Ax \ra\;\mbox{ for all }\;x\in L\;\mbox{ and }\;\la x, Bx \ra \ge -r\|x\|^2\;\mbox{ for all }\;x\in\R^n,
\end{equation}
This allows us to rewrite \eqref{eq:d2fgtd} in the form
\begin{equation}\label{eq:d2fgtd2}
d^2 f(\bar{x}|\bar{v}) = q_B + \delta_L,
\end{equation}
which being combined with \eqref{eq:envd^2} and \eqref{eq:d2fgtd2} gives us
$$\begin{aligned}
d^2 e_{\lambda} f(\bar{x}+\lambda\bar{v}|\bar{v})(w) &= 2d^2\Big[ \frac{1}{2} e_{\lambda} f\Big]\Big(\bar{x} + \lambda\bar{v} \Big|\dfrac{1}{2}\bar{v}\Big)(w)\\
&=2 e_{\lambda} \Big[ \frac{1}{2}d^2 f(\bar{x}|\bar{v}) \Big](w) \\
&= 2\inf\limits_{ x \in \R^n} \Big\{ \frac{1}{2}\la x,Bx \ra + \delta_L(x) + \frac{1}{2\lambda}\|x - w\|^2 \Big\}\\
&= \inf\limits_{ x \in L}\Big\{ \la x,Bx \ra + \frac{1}{\lambda}\|x - w\|^2 \Big\},\quad w\in\R^n.
\end{aligned} 
$$
If $L = \{0\}$, we obviously have
$$
d^2 e_{\lambda} f(\bar{x}+\lambda\bar{v}|\bar{v})(w) = \la 0,B0\ra + \dfrac{1}{\lambda}\|0 - w\|^2 = \dfrac{1}{\lambda}\|w\|^2,\quad w\in\R^n,  
$$
so this is a quadratic form. Otherwise, consider the smooth function
$$
\varphi (x) :=  \la x,Bx \ra + \frac{1}{\lambda}\|x - w\|^2, \quad  x \in L,
$$
Combining $\lambda < 1/r$ with \eqref{eq:assessmentA2} tells us that $\varphi$ is strongly convex on $L$, and thus $\varphi$ admits a unique global minimizer $\tilde{x}$. Then we get by the subdifferential Fermat and elementary sum rules that
$$
0 \in \nabla \varphi(\tilde{x}) + N_L(\tilde{x}) = 2B\tilde{x} + \frac{2}{\lambda}(\tilde{x} - w) + L^\perp,
$$
which can be rewritten in the form
$$
2B\tilde{x} + \frac{2}{\lambda}(\tilde{x} - w) \in -L^\perp = L^\perp.
$$
This ensures that $\tilde{x}$ satisfies the system
 \begin{equation}\label{eq:system2}
 \left\{ \begin{array}{l}
x \in L, \\
2\Big(B + \dfrac{1}{\lambda} I\Big)x - \dfrac{2}{\lambda} w \in L^\perp.  
\end{array} \right.    
\end{equation}
Take any orthogonal basis of $L$, and let $C$ be the matrix whose columns are the vectors of the basis. Since $1/\lambda < r$, it follows from \eqref{eq:assessmentA2} that $B + \frac{1}{\lambda}I$ is a positive-definite matrix. By Lemma \ref{lem:sol}, $\tilde{x}$ is the unique solution to system \eqref{eq:system2} that is given by the formula
$$
\tilde{x} = \frac{1}{\lambda} C\left(C^T\left(B + \frac{1}{\lambda}I \right)C\right)^{-1}C^T w.
$$
Define now the matrices
$$
M :=  \frac{1}{\lambda} C\Big(C^T\Big(B + \frac{1}{\lambda}I\Big)C\Big)^{-1}C^T, \quad Q:= M B M + \frac{1}{\lambda}\Big(M - I\Big)^2
$$
and observe that both $M$ and $Q$ are symmetric and that $\tilde{x}= Mw$. Then we get
$$
\begin{aligned}
d^2 e_{\lambda} f (\bar{x} + \lambda\bar{v}|\bar{v})(w) = \varphi(\tilde{x}) = \la Mw,BMw \ra + \frac{1}{\lambda}\left\|Mw - w\right\|^2 
= \la w, Qw \ra.
\end{aligned}
$$
This tells us that if $f$ is generalized twice differentiable at $\bar{x}$ for $\bar{v}$, then $d^2 e_{\lambda} f(\bar{x} + \lambda \bar{v}|\bar{v})(w)$ is an ordinary quadratic form. It follows from Proposition~\ref{prop:twice-epi}, due to the twice epi-differentiability of $f$ at $\bar{x}$ for $\bar{v}$, that $e_{\lambda} f$ is twice epi-differentiable at $\bar{x} + \lambda\bar{v}$ for $\bar{v}$. Since $d^2 e_{\lambda} f(\bar{x} + \lambda\bar{v}|\bar{v})$ is a quadratic form, $e_\lambda f$ is generalized twice differentiable at $\bar{x} + \lambda \bar{v}$ for $\bar{v}$. Remembering that $f$ is prox-regular at $\bar{x}$ for $\bar{v}$, we conclude that $e_{\lambda} f$ is of class $\mathcal{C}^{1,1}$ around $\bar{x} + \lambda\bar{v}$. Therefore, Proposition~\ref{prop:d2finite} ensures that $e_\lambda f$ is twice differentiable at $\bar{x} + \lambda\bar{v}$.\vspace*{0.05in}

To verify now the reverse implication {\bf (ii)$\Longrightarrow$(i)}, assume that $e_\lambda f$ is twice differentiable at $\bar{x} + \lambda \bar{v}$. As follows from Proposition~\ref{prop:twice-epi}, $f$ is twice epi-differentiable at $\bar{x}$ for $\bar{v}$. It remains to  show that $d^2f(\bar{x}|\bar{v})$ is a generalized quadratic form. Since $\nabla e_{\lambda}f$ is differentiable at $\bar{x} + \lambda\bar{v}$, it is proto-differentiable at $\bar{x} + \lambda\bar{v}$ for $\bar{v}$ by Proposition~\ref{rmk:protosingle_diff}. Then  Lemma~\ref{prop:prox-mor-pr} tells is that 
$$
\nabla e_\lambda f = [\lambda I + T_\ve^{-1}]^{-1}\;\mbox{ in a neighborhood }\;U\;\mbox{ of }\;\bar{x} + \lambda\bar{v}
$$
for some mapping $f$-attentive $\ve$-localization of $\partial f$ around $(\bar{x},\bar{v})$, and hence $[\lambda I + T_\ve^{-1}]^{-1}$ is also proto-differentiable at $\bar{x} + \lambda \bar{v}$ for $\bar{v}$. Since $\nabla e_{\lambda}f$ is differentiable at $\bar{x} + \lambda\bar{v}$, we deduce from \cite[Example 8.34]{Rockafellar98} that the proto-derivative of $\nabla e_{\lambda}f$ at $\bar{x} + \lambda \bar{v}$ for $\bar{v}$ is given by
\begin{equation}\label{eq:D_env}
D[\lambda I + T_\ve^{-1}]^{-1}(\bar{x} + \lambda\bar{v}|\bar{v})(w) = D\nabla e_{\lambda}f(\bar{x} + \lambda\bar{v}|\bar{v})(w) = \nabla^2 e_{\lambda}f(\bar{x} + \lambda\bar{v})w,\quad w\in\R^n.
\end{equation}
It follows now from Lemma~\ref{prop:shear} that
$$
\gph T_\ve = A_{\lambda} [\gph [\lambda I + T^{-1}_\ve]^{-1}]
$$
with the mapping $A_{\lambda}$ defined therein. Applying Lemma~\ref{lem:gphDGF} with $A: = A_{\lambda}$, $F:=[\lambda I + T_\ve^{-1}]^{-1}$, and $G: = T_\ve$ ensures that $T_\ve$ is proto-differentiable at $\bar{x}$ for $\bar{v}$ and
\begin{equation}\label{eq:DT}
\gph DT_\ve(\bar{x}|\bar{v}) = A_\lambda[\gph D [\lambda I + T_\ve^{-1}]^{-1}](\bar{x} + \lambda\bar{v}|\bar{v})].
\end{equation}
By \eqref{eq:D_env}, $\gph D [\lambda I + T_\ve^{-1}]^{-1}](\bar{x} + \lambda\bar{v}|\bar{v})$ is a linear subspace of dimension $n$ in $\R^n \times \R^n$. This implies that $A_\lambda[\gph D [\lambda I + T_\ve^{-1}]^{-1}](\bar{x} + \lambda\bar{v}|\bar{v})]$ is also a linear subspace of dimension $n$ due to the invertibility of $A_{\lambda}$. Combining the latter with \eqref{eq:DT}, we conclude that $\gph DT_\ve(\bar{x}|\bar{v})$ is a $n$-dimensional linear subspace of $\R^n \times \R^n$. Applying finally \cite[Proposition~13.40]{Rockafellar98} tells us that
$$
\partial\Big( \frac{1}{2}d^2 f(\bar{x}|\bar{v})\Big) = DT_\ve(\bar{x}|\bar{v}),
$$
and so the set $\gph \partial\big( \dfrac{1}{2}d^2 f(\bar{x}|\bar{v})\big)$ is a linear subspace of dimension $n$. This means that $\dfrac{1}{2}d^2 f(\bar{x}|\bar{v})$ is a generalized quadratic form and thus completes the proof of the theorem.
\end{proof}

To conclude the section, we provide a uniform characterization of generalized twice differentiability along $f$-attentive $\ve$-localizations of $\partial f$ around the reference point via the classical twice differentiability of the associated Moreau envelopes. The uniformity here refers to the fact that we can fix a single sufficiently small parameter for every points in the aforementioned $f$-attentive localization.

\begin{Corollary}\label{cor:gtdtwicediff} Let $f: \R^n \to \overline{\R}$ be prox-bounded on $\R^n$ and $r$-level 
prox-regular at $\bar{x}$ for $\bar{v} \in \partial f(\bar{x})$. Then for any small $\ve>0$, there exists an $f$-attentive $\ve$-localization $T_\ve$ of $\partial f$ around $(\bar{x},\bar{v})$ such that the following assertions are equivalent for all $\lambda \in (0,1/r)$ and all $(x,v) \in \gph T_\ve$:
\begin{itemize}
\item[{\rm \textbf{(i)}}] $f$ is generalized twice differentiable at $x$ for $v$. 
 
\item[{\rm \textbf{(ii)}}] $e_{\lambda} f$ is twice differentiable at $x + \lambda v$.
\end{itemize}
Furthermore, we have for all $(x,v) \in \gph T_\ve$ that
\begin{equation}\label{eq:envd^22}
 e_{\lambda} \Big[ \frac{1}{2}d^2 f (x|v)\Big] =  d^2\Big[ \frac{1}{2} e_{\lambda} f\Big]\Big(x+\lambda v \Big| \dfrac{1}{2}v\Big).
\end{equation}
\end{Corollary}
\begin{proof}
Due to the imposed prox-boundedness and prox-regularity assumptions, it follows from Proposition~\ref{prop:proxlan} that there exists an $f$-attentive $\ve$-localization $T_\ve$ of $\partial f$ around $(\bar{x},\bar v)$ such that $f$ is $r$-level prox-regular at $x$ for $v$ whenever $(x,v)\in\gph f$. Applying Theorem~\ref{prop:gtdtwicediff} tells us that for all $\lambda \in (0,1/r)$, $f$ is generalized twice differentiable at $x$ for $v$ if and only if $e_{\lambda} f$ is twice differentiable at $x + \lambda v$. Since $f$ is $r$-level prox-regular at $x$ for $v$, the equality in \eqref{eq:envd^22} also follows from \eqref{eq:envd^2} in Theorem~\ref{prop:gtdtwicediff}.
\end{proof}

\section{Quadratic Bundles}\label{sec:quad}

This section addresses the notion of {\em quadratic bundles} for extended-real-valued functions, which involves the two main components: generalized twice differentiability of functions and the epigraphical limits of the corresponding second-order epi-derivatives as generalized quadratic forms.

For any proper l.s.c,\ function $f: \R^n \to \overline{\R}$, consider the set
\begin{equation}\label{eq:tapgtd}
\Omega_f := \big\{(x,v) \in \gph \partial f\;\big|\;f\text{ is generalized twice differentiable at }x\;\text{ for }\;v\big\}.
\end{equation}
As an example, let $f$ be the norm-2 function on $\R^n$, which we already studied from the viewpoint of generalized twice differentiability in Theorem~\ref{prop:normgtd}. Using the calculations therein and taking into account that the dual norm of norm-2 is itself give us the representation
$$
\Omega_f = \big\{ (x,\nabla f(x))\;\big|\;x \ne 0\big\} \cup \big\{ (0,v)\;\big|\;\|v\|< 1\big\}. 
$$
On the other hand, we have
$$
\gph \partial f =\big\{ (x,\nabla f(x))\;\big|\;x \ne 0\big\} \cup \big\{ (0,v)\;\big|\;\|v\|\le 1\big\}.
$$
This shows us that, in this particular example, the set $\Omega_f$ is not necessarily closed while it is {\em dense} in the graph 
$\gph \partial f$. It follows from the results below that this density property holds for every l.s.c.\ {\em convex} function.  Moreover, we show that the set $\Omega_f$ in \eqref{eq:tapgtd} is {\em locally dense} in the graph of $\partial f$ if $f$ is {\em prox-regular}. To proceed in this way, consider the {\em Hessian bundle} of $f$ at a point $\bar{x} \in \dom f$ is defined by
\begin{equation}\label{eq:hessianbundle}
\overline{\nabla}^2 f(\bar{x}):=\big\{ H\in\R^{n \times n}\;\big|\;\exists x_k \to \bar{x} \text{ such that } f \text{ is twice differentiable at } x_k\;\text{ and }\;\nabla^2 f(x_k) \to H \big\}.
\end{equation}
When $f$ is of class $\mathcal{C}^{1,1}$ around $\bar{x}$, the Hessian bundle $\overline{\nabla}^2 f(\bar{x})$ is a nonempty and compact set that consists of symmetric matrices; \cite[Theorem~13.52]{Rockafellar98}.

\begin{Theorem}\label{prop:omegadense} Let $f: \R^n \to \overline{\R}$ be prox-regular at $\bar{x}$ for $\bar{v} \in \partial f(\bar{x})$, and let $\ve>0$. Then we have:
\begin{itemize}
 \item[\rm\textbf{(i)}] There is an $f$-attentive $\ve$-localization $T_\ve$ of $\partial f$ around $(\bar{x},\bar{v})$ such that $\Omega_f \cap \gph T_\ve$ is dense in $\gph T_\ve$.
 
\item[\rm\textbf{(ii)}] If in addition $f$ is subdifferentially continuous at $\bar{x}$ for $\bar{v}$, then there exists a neighborhood $W$ of $(\bar{x},\bar{v})$ such that the set $\Omega_f \cap W$ is dense in $W \cap \gph \partial f$.
 
\item[\rm\textbf{(iii)}] If $f$ is convex, then the aforementioned property becomes global, i.e., $\Omega_f$ is dense in $\gph \partial f$.
\end{itemize}
\end{Theorem}
\begin{proof}
\textbf{(i)} By Proposition~\ref{prop:proxlan}, there exists an $f$-attentive $\ve$-localization $T_\ve$ of $\partial f$ around $(\bar{x},\bar{v})$ such that $f$ is prox-regular at $x$ for $v$ whenever $(x,v) \in \gph 
T_\ve$. Fix $(x,v) \in \gph T_\ve$ and deduce from  the prox-regularity of $f$ at $x$ for $v$ by using Lemma~\ref{prop:prox-mor-pr} that there exists $\lambda > 0$ so small that $e_{\lambda} f$ is of class $\mathcal{C}^{1,1}$ around $x + \lambda v$. Furthermore, we have on a neighborhood of $x + \lambda v$ that
\begin{equation}\label{eq:gradenv1}
\nabla e_{\lambda} f = [\lambda I + (T'_\ve)^{-1}]^{-1},
\end{equation}
for an $f$-attentive $\ve$-localization $T'_\ve$ of $\partial f$ around $(x,v)$. It follows from Proposition~\ref{prop:proxlan} that we can shrink $\gph T'_\ve$ to get $\gph T'_\ve\subset \gph T_\ve$.
By \cite[Theorem~13.52]{Rockafellar98}, the Hessian bundle $\overline{\nabla}^2 e_{\lambda}f (x + \lambda v)$ is nonempty. Pick any $H \in \overline{\nabla}^2 e_{\lambda}f (x + \lambda v)$ and find a sequence $z_k \to x + \lambda v$ such that $e_{\lambda} f$ is twice differentiable in the classical sense at $z_k$ and $\nabla^2 e_{\lambda} f (z_k) \to H$ as $k\to\infty$. Fix $k\in\N$ and define
$$
v_k := \nabla e_\lambda f(z_k), \quad x_k := z_k - \lambda v_k.
$$
By Lemma~\ref{lem:zkvkxk}, we have $v_k \in T'_\ve(x_k) \subset T_\ve(x_k)$ for all large $k$ with $v_k \to v$ and $x_k \to x$ as $k\to\infty$. Since $e_{\lambda} f$ is twice differentiable at $z_k$, 
Theorem~\ref{prop:gtdtwicediff} tells us that $f$ is generalized twice differentiable at $x_k$ for $v_k$ meaning that $(x_k,v_k) \in \Omega_f$. Therefore, the sequence $(x_k,v_k)$ belongs to the set $\gph T_\ve \cap \Omega_f$ for all large $k$ and converges to $(x,v)$. This readily implies that the set $\gph T_\ve \cap \Omega_f$ is dense in $\gph T_\ve$ as claimed.
\medskip

\textbf{(ii)} Assume in addition that $f: \R^n \to \overline{\R}$ is subdifferentially continuous at $\bar{x}$ for $\bar{v}$. Combining this with Proposition~\eqref{prop:proxlan} allows us to find a neighborhood $W$ of $(\bar{x},\bar{v})$ such that $f$ is prox-regular at $x$ for $v$ $(x,v) \in \gph T_\ve$. Then we come to the claimed conclusion by repeating the arguments in the proof of \textbf{(i)}.

\medskip

\textbf{(iii)} When $f$ is convex, we can utilize a stronger version of Lemma~\ref{prop:prox-mor-pr}, where \eqref{eq:gradenv1} holds everywhere with $T'_\ve$ replaced by $\partial f$; see \cite[Theorem~2.26]{Rockafellar98}.  This brings us to the desired conclusion and therefore completes the proof of the theorem.
\end{proof}

Our next goal is to study {\em epi-convergence} of sequences of generalized quadratic forms. Observe that if a sequence of linear subspaces $L_k$ converges, then it converges to a linear subspace $L$. Indeed, it follows from \cite[Exercise~4.14 and Proposition~4.15]{Rockafellar98} that $L$ is a convex cone. Since $L_k = -L_k$ for all $k$, we get $-L = L$ as claimed. The following proposition presents the facts needed in our subsequent analysis; its second part is also mentioned in \cite[Proposition~4.3]{roc24}.

\begin{Proposition}\label{prop:epiconv_gqf}  Let $\{q_k\}$ be a sequence of generalized quadratic forms. Then we have:
\begin{itemize}
\item[\rm\textbf{(i)}] We can extract a subsequence of $\{q_k\}$ that epigraphically converges to $q \not\equiv\infty$.
 
\item[\rm\textbf{(ii)}] Assume that $\{q_k\}$ epigraphically converges to $q$ and there exists a number $r \ge 0$ such that
\begin{equation}\label{eq:qkge-r}
q_k(w) \ge -r \|w\|^2,\quad w\in\R^n,  \end{equation}
for all sufficiently large $k$. Then $q$ is a generalized quadratic form  satisfying
$$ 
q(w) \ge-r\|w\|^2\;\mbox{ whenever }\;w\in\R^n.
$$
\end{itemize}
\end{Proposition}
\begin{proof}
\textbf{(i)} Since each function
$q_k$ is a generalized quadratic form, we have $q_k(0) = 0$ for all $k\in\N$. It follows from 
Proposition~\ref{prop:epichar} that 
$$
\Big( \eliminf\limits_{k \to \infty} q_k\Big)(0) \le \liminf\limits_{k \to \infty} q_k(0) = 0.
$$
Hence $\eliminf\limits_{k \to \infty} q_k \not\equiv\infty$, which shows that the sequence $\{q_k\}$ does not escape epigraphically to the horizon, i.e., $\epi q_k \not\rightarrow \emptyset$ by \cite[Proposition~7.5]{Rockafellar98}. This ensures by \cite[Proposition~7.6]{Rockafellar98} that $\{q_k\}$ contains a subsequence that epigraphically converges to a function $q \not\equiv\infty$.

\medskip

\textbf{(ii)} Let us show that the assumptions on $\{q_k\}$ and $q$ yield $q(0) = 0$. Indeed, for any $w_k \to 0$ we have
$$
\liminf\limits_{k \to \infty} q_k(w_k) \ge \liminf\limits_{k \to \infty} (-r\|w_k\|^2) = 0, 
$$
which implies that $q_k(w_k) = 0$ when $w_k = 0$, $k\in\N$, and thus $q(0)=0$ by Proposition~\ref{prop:epichar}. Define further the sequence of generalized quadratic forms $\{p_k\}$ by
$$
p_k(w) := q_k(w) + r\|w\|^2,\quad w\in\R^n.
$$
We obviously have that all $p_k$ are generalized quadratic form, and it follows from \eqref{eq:qkge-r} that $p_k(w) \ge 0$ whenever $w \in \R^n$. Therefore, each $p_k$ is convex for large $k$. Applying further \cite[Theorem~7.46(ii)]{Rockafellar98} ensures that the sequence $\{p_k\}$ epigraphically converges to $q + r\|\cdot\|^2$. Furthermore, \cite[Theorem~12.35]{Rockafellar98} tells us that $\{\partial p_k\}$ graphically converges to $\partial(q + r\|\cdot\|^2) = \partial q + 2rI$. Since each set $\gph \partial q_k$ is a linear subspace in $\R^n \times \R^n$ for all $k$, the limiting one $\gph(\partial q + 2rI)$ is also a subspace, and so is $\gph \partial q$. By $q(0) = 0$, this amounts to $q$ being a generalized quadratic form.
\end{proof}

Now we are in a position to present the definition of {\em quadratic bundles} of extended-real-valued functions, which was first introduced in \cite[p.\ 187]{roc} in order to work with convex functions. Here we revise this concept when moving to the general case. Note that when losing the innate subdifferential continuity of convex functions, we need to provide an appropriate modification to obtain a reasonable construction. In the original work, Rockafellar takes all the generalized quadratic forms that are epigraphical limits of a sequence of second-order subderivatives, where the primal-dual pair tends to a specified pair while making sure that the function is generalized twice differentiable all along. In the process of this work, we learn from the very recent manuscript \cite{roc24}, where a generalized version of quadratic bundles is proposed for nonconvex functions with incorporating the $f$-attentiveness. This change is crucial to produce effective results as given below.

\begin{Definition}\rm \label{defi:quadbund}  Let $f: \R^n \to \overline{\R}$ be a proper extended-real-valued function. The \textit{quadratic bundle} of $f$ at $\bar{x} \in \dom f$ for $\bar{v} \in \partial f(\bar{x})$, denoted by $\mathrm{quad} f (\bar{x}|\bar{v})$, is defined as the collection of generalized quadratic forms $q$ for which there exists $(x_k,v_k) \xrightarrow{\Omega_f} (\bar{x},\bar{v})$ such that $f(x_k) \to f(\bar{x})$ and the sequence of generalized quadratic forms $q_k = \frac{1}{2} d^2 f(x_k|v_k)$ converges epigraphically to $q$.
\end{Definition}

The following remark collects some important observations about quadratic bundles.

\begin{Remark}\label{prop:extractepi} \rm \textbf{}
\begin{itemize}
\item[\textbf{(i)}] When $f$ is convex (or more generally, when $f$ is subdifferentially continuous at $\bar{x}$ for $\bar{v}$), the revised definition of $\qua f(\bar{x}|\bar{v})$ agrees with the original construction by Rockafellar in \cite[p.\ 187]{roc}.

\item[\textbf{(ii)}] It is necessary to impose the additional condition $f(x_k) \to f(\bar{x})$ for the primal-dual sequence $\{(x_k,v_k)\}$, because the original construction fails to capture the strong variational convexity in the absence of subdifferential continuity. Various results and discussions in this direction are given in our forthcoming manuscript \cite{quadcharvar}.

\item[\textbf{(iii)}] In Definition~\ref{defi:quadbund}, along with the requirements imposed on the sequences  $\{(x_k,v_k)\}$, the functions $\dfrac{1}{2}d^2 f(x_k|v_k)$ need to be epigraphically convergent as well. Actually, as long as we can find a sequence of $(x_k,v_k)$  such that $(x_k,v_k) \xrightarrow{\Omega_f} (\bar{x},\bar{v})$ and  $f(x_k) \to f(\bar{x})$, we can assume without loss of generality that the functions $\dfrac{1}{2}d^2 f(x_k|v_k)$ epigraphically converge without any additional information. In other words, the construction of quadratic bundles relies solely on the sequences $\{(x_k,v_k)\}$ that fulfill the imposed requirements. This observation follows from Proposition~\ref{prop:epiconv_gqf}.

\item[\textbf{(iv)}] Suppose that $f: \R^n \to \overline{\R}$ is prox-regular at $\bar{x}$ for $\bar{v} \in \partial f(\bar{x})$. Then whenever $(x_k,v_k) \xrightarrow{\Omega_f} (\bar{x},\bar{v})$ and $f(x_k) \to f(\bar{x})$, we see that if the functions $\dfrac{1}{2}d^2 f(x_k|v_k)$ epigraphically converge, then they converge to a generalized quadratic form $q$. In other words, $\qua f(\bar{x}|\bar{v})$ contains all the possible epi-limits $q$ of $\dfrac{1}{2}d^2 f(x_k|v_k)$ with $(x_k,v_k) \xrightarrow{\Omega_f}(\bar{x},\bar{v})$ and $f(x_k) \to f(\bar{x})$ as $k\to\infty$. To verify this, suppose that $f: \R^n \to \overline{\R}$ is $r$-level prox-regular at $\bar{x}$ for $\bar{v}$ with some $r \ge 0$. It follows from Proposition~\ref{prop:proxlan} that $f$ is $r$-level prox-regular at $x$ for $v$ whenever $(x,v)$ belongs to the graph of an $f$-attentive $\ve$-localization $T_\ve$ of $\partial f$ around $(\bar{x},\bar{v})$.  Moreover, Corollary~\ref{cor:d2prox} tells us that 
$$
d^2 f(x|v)(w) \ge -r\|w\|^2\;\mbox{ whenever }\;w \in \R^n
$$ 
for all such pairs of $(x,v)$. Take further any sequences $(x_k,v_k) \xrightarrow{\Omega_f} (\bar{x},\bar{v})$ and $f(x_k) \to f(\bar{x})$ such that the functions $\frac{1}{2}d^2 f(x_k|v_k)$ epigraphically converge to $q$. Then $(x_k,v_k) \in \gph T_\ve$ for all $k$ large enough. Consequently, we have 
$$
d^2 f(x_k|v_k)(w) \ge -r\|w\|^2,\quad w \in \R^n.
$$ 
Applying Proposition~\ref{prop:epiconv_gqf}\textbf{(ii)} confirms that $q$ is a generalized quadratic form.

\item[\textbf{(v)}] If$f$ is generalized twice differentiable at $\bar{x}$ for $\bar{v}$, then
$$
\dfrac{1}{2} d^2 f(\bar{x}|\bar{v})  \in \qua f(\bar{x}|\bar{v}).
$$
Indeed, let $(x_k,v_k): = (\bar{x},\bar{v})$ for all $k$. It is clear that $f$ is generalized twice differentiable at $x_k$ for $v_k$ and that the sequence of $\frac{1}{2}d^2 f (x_k | v_k)$ converges epigraphically to $\frac{1}{2}d^2 f (\bar{x} | \bar{v})$, which readily shows that $\frac{1}{2}d^2 f(\bar{x}|\bar{v}) \in\qua f(\bar{x}|\bar{v})$. 
\end{itemize}
\end{Remark}

If $f$ is subdifferentially continuous at $\bar{x}$ for $\bar{v}$, then the requirement $f(x_k) \to f(\bar{x})$ in 
Definition~\ref{defi:quadbund} is superfluous. However, the following example shows that our modification in the absence of subdifferential continuity makes a difference. For convenience, we call the construction of quadratic bundles, where the addition condition $f(x_k) \to f(\bar{x})$ is not imposed, the ``old" quadratic bundle with the notation $\qua_o f$. This means that for any $(\bar{x},\bar{v}) \in \gph \partial f$, the ``old" $\qua_o f(\bar{x}|\bar{v})$ is defined as the collection of generalized quadratic forms $q$ for which there exists $(x_k,v_k) \xrightarrow{\Omega_f} (\bar{x},\bar{v})$ such that the sequence of generalized quadratic forms $q_k = \frac{1}{2} d^2 f(x_k|v_k)$ converges epigraphically to $q$. 

\begin{Example}\rm \label{exam:quadandquads} {\em Consider the function $f: \R \to \R$ defined by
$$
f(x): = \left\{\begin{array}{ll}
x^2,  & x \ge 0,  \\
1,  & x < 0.
\end{array}\right.
$$
Then $\partial f(0) = (-\infty,0]$, $f$ is variationally $2$-strongly convex at $0$ for $0$ $($in particular, it is prox-regular at $0$ for $0)$ but not subdifferentially continuous at $0$ for $0$. Moreover, we have
\begin{equation}\label{eq:quadandquads}
\qua_o f(0|0) = \{ q_{[0]},  q_{[1]}, \delta_{\{0\}}\}, \quad  \qua f(0|0) = \{q_{[1]}, \delta_{\{0\}}\}.
\end{equation}}
\end{Example} 
\begin{proof} First we check that the function $f$ is variationally $2$-strongly convex by Definition~\ref{defi:varconv}. Indeed, direct calculations bring us
$$
\widehat{\partial}f(x) = \partial f(x) = \left\{ \begin{array}{ll}
\{2x\},   & x > 0, \\
(-\infty,0],   & x = 0,\\
\{ 0\}, & x < 0.
\end{array} \right. 
$$
Take $\varepsilon:= 1/4$ together with the neighborhoods $U := (-1/2,1/2)$ of $0$ and $V := (-1/2,1/2)$ of $0$. Then
\begin{equation}\label{eq:fattentive-partialf}
(U_\varepsilon \times V) \cap \gph \partial f =\big\{ (0,v)\;\big|\;v \in (-1/2,0]\big\} \cup\big\{(u,2u)\;\big|\; u \in (0,1/4)\big\}. 
\end{equation}
Consider further the function $\widehat{f} : \R \to \R$ defined by
$$
\widehat{f}(x): =\left\{ \begin{array}{ll}
x^2,  & x \ge 0,  \\
x^2 - x,  & x < 0.
\end{array} \right. 
$$
It is straightforward to check that
$$
\widehat{\partial}\widehat{f}(x) = \partial \widehat{f}(x) = \left\{\begin{array}{ll}
\{2x\},   & x > 0, \\
\text{[}-1,0],   & x = 0,\\
\{ 2x - 1\}, & x < 0.
\end{array} \right. 
$$
Consequently, we have the expression
\begin{equation}\label{eq:fattentive-partialf2}
(U \times V) \cap \gph \partial \widehat{f} =\big\{ (0,v)\;\big|\;v \in (-1/2,0]\big\} \cup\big\{(u,2u)\;\big|\;u \in (0,1/4)\big\}. 
\end{equation}
It follows from \eqref{eq:fattentive-partialf} and \eqref{eq:fattentive-partialf2} that
$$
(U \times V) \cap \gph \partial \widehat{f} = (U_\varepsilon \times V) \cap \gph \partial f
$$ 
for the chosen $\varepsilon$, $U$, and $V$. Moreover, for all $x \ge 0$ we get $\widehat{f}(x) = f(x) = x^2$ and $\widehat{f} \le f$ on $U$. Since $\widehat{f}(x) - x^2$ is convex, the function $\widehat{f}(x)$ is $2$-convex, and thus $f$ is variationally $2$-strongly convex by definition. On the other hand, $f$ is not subdifferentially continuous at $0$ for $0$. To see this, take a sequence $\{(x_k,0)\}$ with $x_k \uparrow 0$, which fulfills the condition $(x_k,0) \xrightarrow{\gph \partial f}(0,0)$ while $f(x_k) \not\to f(0)$ as $k\to\infty$.\vspace*{0.05in} 

To proceed with the verification of our claims in this example, now we provide the calculations of the quadratic bundle and ``old" quadratic bundle of $f$ at $0$ for $0$. Let us first calculate $d^2 f(x|v)$ for all pairs $(x,v) \in \gph \partial f$. We see that
$$
\Delta_t^2 f(0|v)(w) = \dfrac{f(0 + tw) - f(0) - vtw}{\frac{1}{2}t^2}  = \left\{ \begin{array}{ll}
\dfrac{2-2tvw}{t^2},   & w < 0,  \\
2w^2 - \dfrac{2vw}{t},& w \ge 0,
\end{array} \right.
$$
for all $v \in (-\infty,0]$, all $t > 0$, and all $w \in \R^n$.
It easily follows that
$$
d^2 f(0|v)(w) = \liminf\limits_{\substack{t \downarrow 0\\ w' \to w}} \Delta_t^2 f(0|v)(w') = \left\{\begin{array}{ll}
\delta_{\{0\}}(w), &  v < 0,  \\ 
\delta_{[0,\infty)}(w) + 2w^2,   & v = 0, 
\end{array}  \right. 
$$
whenever $w\in\R^n$. Observe that $f$ is twice differentiable at all $x \in \R \setminus \{0\}$,  and so 
$$
d^2 f(x|f'(x))(w) = \left\{ \begin{array}{ll}
2w^2, & x > 0,  \\
 0, & x < 0,
\end{array}\right.
$$
for all $x \in \R \setminus \{0\}$ all $w \in \R$, and all $w\in\R$. Next we determine the set $\Omega_f$ from \eqref{eq:tapgtd}, which contains all the pairs $(x,v) \in \gph \partial f$ where $f$ is generalized twice differentiable at $x$ for $v$. To this end, observe that $f$ is not generalized twice differentiable at $0$ for $0$ because $d^2 f(0|0)$ is not a generalized quadratic form. Since $f$ is twice differentiable at all $x \ne 0$, $f$ is also generalized twice differentiable at $x$ for $f'(x)$ for all $x \ne 0$. When $v < 0$,  Proposition~\ref{prop:epichar} tells us that
$$
\Big(\elimsup\limits_{t \downarrow 0} \Delta_t^2 f(0|v)\Big)(w) = \delta_{\{0\}}(w),\quad w\in\R^n,
$$
which entails the twice epi-differentiability of $f$ at $0$ for such $v$. Since the function $d^2 f(0|v)$ is also a generalized quadratic form, $f$ is generalized twice differentiable at $0$ for $v$ for all $v < 0$. Therefore, the set $\Omega_f$ in \eqref{eq:tapgtd} is calculated by
$$
\Omega_f =\big\{ (0,v)\;\big|\; v < 0\big\}\cup\big\{ (x,f'(x))\;\big|\;x \ne 0\big\}. 
$$
To determine further the ``old" quadratic bundles, pick $q \in \qua_o f(0|0)$ and find $(x_k,v_k) \xrightarrow{\Omega_f} (0,0)$ and $\dfrac{1}{2}d^2 f(x_k|v_k) \xrightarrow{e} q$ as $k\to\infty$. Consider the following three possibilities:
\begin{itemize}
\item[\textbf{(i)}] If $x_k = 0$ for infinitely many indexes of $k$, then there exists a subsequence $k_m$ such that $x_{k_m} = 0$ for all $m$, and so $v_{k_m} < 0$ for all $m$. Consequently, 
$$
\dfrac{1}{2}d^2 f(x_{k_m}|v_{k_m}) = \delta_{\{0\}} \xrightarrow{e} \delta_{\{0\}},
$$
and hence we have $q = \delta_{\{0\}}$.

\item[\textbf{(ii)}] If $x_k = 0$ for finitely many indexes and there are infinitely many indexes $k$ with $x_k > 0$, take a subsequence $\{k_m\}$ such that $x_{k_m} > 0$. Then $v_{k_m} = f'(x_{k_m})$ and $\dfrac{1}{2}d^2 f(x_{k_m}|v_{k_m}) = q_{[1]} \xrightarrow{e} q_{[1]}$, which  therefore yield $q = q_{[1]}$.

\item[\textbf{(iii)}] If $x_k \ge 0$ for finitely many indexes and $x_k < 0$ for infinitely many indexes, then there exists a subsequence $\{k_m\}$ such that $x_{k_m} < 0$ for all $m$. Then $v_{k_m} = f'(x_{k_m})$ and $\dfrac{1}{2}d^2 f(x_{k_m}|v_{k_m}) = q_{[0]} \xrightarrow{e} q_{[0]}$, which verify that $q = q_{[0]}$. 
\end{itemize}
Combining all the above allows us to conclude that
$$
\qua_o f(0|0) \subset \{ \delta_{\{0\}}, q_{[0]}, q_{[1]} \}.
$$
It is straightforward to check the reverse inclusion, which shows that the first part of \eqref{eq:quadandquads} holds. Finally, note that if we impose in addition the condition $f(x_k) \to f(0)$, then it follows that $x_k \ge 0$ for all large $k$. Then we continue with the arguments similar to  the parts \textbf{(i)} and \textbf{(ii)}, which brings us to the inclusion
$$
\qua f(0|0) \subset\big\{ \delta_{\{0\}}, q_{[1]}\big\}. 
$$
The opposite inclusion can be checked easily, which readily verifies the second part of \eqref{eq:quadandquads}.
\end{proof}

Our next result establishes a {\em sum rule} for quadratic bundles.

\begin{Theorem}\label{prop:sumrulequad}
Assume that $f: \R^n \to \overline{\R}$ is $\mathcal{C}^2$-smooth around $\bar{x}$ and that $g: \R^n \to \overline{\R}$ is proper l.s.c.\ Then for any $\bar{v} \in \partial g(\bar{x})$, we have $\nabla f(\bar{x}) + \bar{v} \in \partial (f+g)(\bar{x})$ and
\begin{equation*}
\qua (f+g)(\bar{x}|\nabla f(\bar{x}) + \bar{v}) = \frac{1}{2} q_{\nabla^2 f(\bar{x})} + \qua g(\bar{x} | \bar{v}). 
\end{equation*}
In particular, the quadratic bundle of $f$ is calculated by
$$
\qua f(\bar{x}|\nabla f(\bar{x})) = \Big\{ \frac{1}{2} d^2 f(\bar{x}|\nabla f (\bar{x}))\Big\} = \Big\{\frac{1}{2} q_{\nabla^2 f(\bar{x})}\Big\}.
$$
\end{Theorem}
\begin{proof} We easily get (see, e.g., \cite[Proposition~1.30]{Mordukhovich18}) that $\nabla f(\bar{x}) + \bar{v} \in \partial(f + g)(\bar{x})$. Pick 
$$
q \in \qua (f+g)(\bar{x}|\nabla f(\bar{x}) + \bar{v})
$$ 
and find $(x_k,v_k) \xrightarrow{\Omega_{f+g}} (\bar{x},\nabla f(\bar{x}) + \bar{v})$ with $f(x_k) + g(x_k) \to f(\bar{x}) + g(\bar{x})$ and  $\frac{1}{2}d^2 (f+g)(x_k | v_k) \xrightarrow{e} q$ as $k\to\infty$. Since $v_k \in \partial (f+g)(x_k)$, the function $f$ is $\mathcal{C}^2$-smooth around $x_k$ for large $k$, and so we get $v_k - \nabla f(x_k) \in \partial g(x_k)$. Remembering that $f + g$ is generalized twice differentiable at $x_k$ for $v_k$ and that $f$ is twice differentiable at $x_k$ for large $k$, we apply Proposition~\ref{prop:sumC2gendif} to conclude that $g = (f + g) + (-f)$ is generalized twice differentiable at $x_k$ for $v_k - \nabla f(x_k)$ and that
\begin{equation}\label{eq:quadsumC2gen3}
d^2 g(x_k|v_k - \nabla f(x_k)) = d^2 (f + g)(x_k|v_k) - \dfrac{1}{2}q_{\nabla^2 f(x_k)}.
\end{equation}
Hence $(x_k,v_k - \nabla f(x_k)) \xrightarrow{\Omega_g} (\bar{x},\bar{v})$ as $k\to\infty$. Since $d^2(f+g)(x_k| v_k) \xrightarrow{e} q$ and $-\frac{1}{2} q_{\nabla^2 f(x_k)} \xrightarrow{c} -\frac{1}{2} q_{\nabla^2 f(\bar{x})}$, it follows from \cite[Theorem~7.46]{Rockafellar98} that
$$
d^2 g(x_k | v_k - \nabla f(x_k)) \xrightarrow{e} q - \frac{1}{2} q_{\nabla^2 f(\bar{x})}.
$$
Furthermore, we see that $g(x_k) \to g(\bar{x})$ and thus deduce from Definition~\ref{defi:quadbund} that $q -\frac{1}{2}q_{\nabla^2 f(\bar{x})} \in \qua g(\bar{x} | \bar{v})
$. In other words, this means that
$$
q \in \dfrac{1}{2}q_{\nabla^2 f(\bar{x})} + \qua g(\bar{x}|\bar{v}),
$$
which readily brings us to the inclusion
\begin{equation}\label{eq:quadsumC2gen1}
\qua (f+g)(\bar{x} | \bar{v} + \nabla f(\bar{x})) \subset \frac{1}{2} q_{\nabla^2 f(\bar{x})} + \qua g(\bar{x} | \bar{v}).
\end{equation}
Take $q$ such that $q \in \frac{1}{2} q_{\nabla^2 f(\bar{x})} + \qua g(\bar{x} | \bar{v})$ telling us that 
\begin{equation}\label{eq:quadsumC2gen2}
q - \frac{1}{2} q_{\nabla^2 f(\bar{x})} \in \qua g(\bar{x} | \bar{v}).
\end{equation}
Representing $g =  (-f) + (f+g)$, we deduce from \eqref{eq:quadsumC2gen2} and \eqref{eq:quadsumC2gen1} that
$$
q - \frac{1}{2} q_{\nabla^2 f(\bar{x})} \in -\frac{1}{2}q_{\nabla^2 f(\bar{x})} + \qua (f+g)(\bar{x} | \bar{v} + \nabla f(\bar{x})),
$$
The latter can be rewritten in the form
\begin{equation*}
 \frac{1}{2} q_{\nabla^2 f(\bar{x})} + \qua g(\bar{x} | \bar{v}) \ra \subset  \qua (f+g)(\bar{x} | \bar{v} + \nabla f(\bar{x})),
\end{equation*}
which completes verifying the claimed sum rule. 
\end{proof}

It is noteworthy that if $f$ is merely twice differentiable at $\bar{x}$ in the classical sense (but not $\mathcal{C}^2$-smooth), then it does not guarantee that $\qua f(\bar{x}|\nabla f(\bar{x}))$ be a singleton. Moreover, the quadratic bundle can consist of uncountable many elements. The following example justifies this observation.

\begin{Example}\label{exam:quadtwicediff}
Consider the function $f: \R \to \R$ defined by
$$
f(x): = \left\{  \begin{array}{ll}
x^4 \sin (1/x), & x \ne 0,  \\
 0, & x = 0.
\end{array}\right. 
$$
Then $f$ is twice differentiable at $0$ in the classical sense, but
\begin{equation}\label{eq:quadexam1}
\qua f(0|0) = \left\{ q_{[a]} \mid a \in [-1/2,1/2] \right\}.
\end{equation}
\end{Example}
\begin{proof} While $f'(0) = f''(0) = 0$, for all $x \ne 0$ we have
$$
f'(x) = x^2\big(-\cos(1/x) + 4 x \sin(1/x)\big), \quad f''(x) =  -6 x \cos(1/x) + (-1 + 12 x^2) \sin(1/x).
$$
Furthermore, it follows from the construction of $f$ that
$$
d^2 f(x|f'(x))(w) = f''(x)w^2\;\mbox{ for all }\;x \in \R\;\mbox{ and }\; w \in \R. 
$$
As $f''$ is bounded around $0$, for any sequence $\{x_k\}$ we have that the epi-convergence of the quadratic functions $f_k(w) = \frac{1}{2}f''(x_k)w^2$ is equivalent to their pointwise convergence. Indeed, assume that $f_k(w) = \frac{1}{2}f''(x_k)w^2$ epigraphically converges to $f$. Fix any $w \in \R^n$ and find by Proposition~\ref{prop:epichar} a sequence $w_k \to w$ such that $f_k(w_k) \to f(w)$. Consequently, we have
$$
\begin{aligned}
|f_k(w) - f(w)| &\le |f_k(w) - f_k(w_k)| + |f_k(w_k) - f(w)|\\
&= \dfrac{1}{2}\left|f''(x_k)\right|.|w_k - w|.|w_k + w| +  |f_k(w_k) - f(w)| \to 0,
\end{aligned}
$$
which entails the pointwise convergence of $f_k$ to $f$. Reversely, suppose that the sequence of $f_k(w) = \frac{1}{2}f''(x_k)w^2$ pointwise converges to $f$. Choose any $w \ne 0$ to see that $f''(x_k) \to a$ for some $a \in \R$ and thus $f(w) = \frac{1}{2}aw^2$.  It is easily to deduce from Proposition~\ref{prop:epichar} $f_k$ epigraphically converges to $f$. Therefore, all the epigraphical limits of such sequences must be quadratic forms $q_{[a]}$, where $\frac{1}{2}f''(x_k) \to a$.\vspace*{0.03in}

To proceed further, fix any $a \in [-1/2,1/2]$ and take $x_k \to 0$ such that $\sin(1/x_k) = -2a$ for all $k$. Then
$$
f''(x_k) = -6x_k \cos(1/x_k) + (-1 + 12x_k^2) \sin(1/x_k) \to 2a,
$$
which shows that the function $q_{[a]}$ belongs to $\qua f(0|0)$. Pick any $q \in \qua f(0|0)$ and deduce from the twice differentiable of $f$ in the classical sense on $\R$ and the boundedness of $f''$ around $0$ that there exists a sequence $x_k \to 0$ such that the sequence of $q_k = \frac{1}{2}q_{[f''(x_k)]}$ epigraphically converges to $q_{[a]}$, $a \in \R$. It follows from our observation above that $\frac{1}{2}f''(x_k) \to a$. Therefore,
$$
2|a| = \limsup |f''(x_k)| \le \limsup \left( 6|x_k|\cdot|\cos(1/x_k)| + 12x_k^2.|\sin(1/x_k)| + |\sin(1/x_k)| \right) \le 1,
$$
and so $a \in [-1/2,1/2]$. This confirms that the quadratic bundle of $f$ at $0$ for $0$ is calculated  by \eqref{eq:quadexam1}. 
\end{proof}

As shown by Example~\ref{exam:quadtwicediff}, the merely twice differentiability of a function everywhere is not sufficient to conclude that the quadratic bundle reduces to the quadratic form associated with the Hessian. To reach such a conclusion in \eqref{eq:quadexam1}, we converted the epi-convergence of quadratic forms into the pointwise convergence of the corresponding Hessians. This suggests that there might be a connection between the quadratic bundle and the Hessian bundle defined as in \eqref{eq:hessianbundle} for classes of functions where the second-order subderivative $d^2 f(x|v)$ is finite. Indeed, the next theorem reveals the relationship between the Hessian bundle of $f$ and the quadratic bundles of $f$ for functions of class $\mathcal{C}^{1,1}$.

\begin{Theorem} Let $f: \R^n \to \overline{\R}$ be an l.s.c.\ function with $\bar{x} \in \dom f$ and $\bar{v} \in \partial f(\bar{x})$. Suppose that $f$ is subdifferentially continuous at $\bar{x}$ for $\bar{v}$. Then we have
\begin{equation}\label{eq:hessianquad}
\qua f(\bar{x}|\bar{v} ) \supset \Big\{ \frac{1}{2}q_H\;\Big|\; H \in \overline{\nabla}^2 f(\bar{x})\Big\}. 
\end{equation}
If moreover $f$ is of class $\mathcal{C}^{1,1}$ around $\bar{x}$ then \eqref{eq:hessianquad} holds as an equality.
\end{Theorem}
\begin{proof} Pick $H \in \overline{\nabla}^2 f(\bar{x})$ and find a sequence $\{x_k\}$ such that $f$ is twice differentiable at $x_k$ in the classical sense for all $k$, $x_k \to \bar{x}$, and $\nabla^2 f(x_k) \to H$ as $k\to\infty$. This implies that $\{q_{\nabla^2 f(x_k)}\}$ epigraphically converges to $q_H$.  Since $f$ is subdifferentially continuous at $\bar{x}$ for $\bar{v}$, the condition $f(x_k) \to f(\bar{x})$ automatically satisfies. It follows from \cite[Proposition~13.8]{Rockafellar98} that $\frac{1}{2}d^2 f(x_k|\nabla f(x_k)) = \frac{1}{2}q_{\nabla^2 f(x_k)}$, and so
$$
\frac{1}{2}q_H \in \qua f(\bar{x}|\bar{v})\;\mbox{ as }\;k\to\infty.
$$
Now assume that $f$ is of class $\mathcal{C}^{1,1}$ around $\bar{x}$. Taking any $q \in \qua f(\bar{x}|\nabla f(\bar{x}))$, find a sequence $(x_k,v_k) \xrightarrow{\Omega_f} (\bar{x},\nabla f(\bar{x}))$ such that $\frac{1}{2}d^2 f(x_k|v_k) \xrightarrow{e} q$, which ensures that $f$ is generalized twice differentiable at $x_k$ for $v_k$ whenever $k\in\N$. Proposition~\ref{prop:d2finite} tells us that $f$ is twice differentiable at $x_k$ in the classical sense, and then we get by \cite[Proposition~13.8]{Rockafellar98} that
$$\frac{1}{2}d^2 f(x_k|v_k)(w) = \frac{1}{2} \la w, \nabla^2 f(x_k)w \ra\;\mbox{ for all }\;w\in\R^n.
$$
The $\mathcal{C}^{1,1}$ property of $f$ around $\bar{x}$ yields the boundedness of $\nabla^2 f$ around this point, while 
\cite[Theorem~13.51]{Rockafellar98} ensures that $\nabla^2 f(x_k)$ is symmetric. Also it easily follows from Proposition~\ref{prop:epichar} that $q$ is finite everywhere. Since functions of class $\mathcal{C}^{1,1}$ are prox-regular, we deduce from Remark~\ref{prop:extractepi}\textbf{(ii)} that $q$ is an ordinary quadratic form denoted by $\frac{1}{2}q_H$, where the matrix $H$ is symmetric. It remains to show that 
$\{\nabla^2 f(x_k)\}$ converges to $H$, which ensures that $H \in \overline{\nabla}^2 f(\bar{x})$. Indeed, since $\frac{1}{2}q_{\nabla^2 f(x_k)} \xrightarrow{e} \frac{1}{2}q_H$, we deduce from Proposition~\ref{prop:epichar} that there exists a sequence $w_k \to w$ with 
$$
\frac{1}{2} \la w_k, \nabla^2 f(x_k)w_k \ra = q_{\nabla^2 f(x_k)}(w_k)\to q_H(w) = \frac{1}{2}\la w, Hw \ra,\quad w\in\R^n.
$$
To this end, observe that
$$
\la w_k, \nabla^2 f(x_k)w_k \ra - \la w, Hw \ra = \la w_k - w, \nabla^2 f(x_k) (w_k + w) \ra + \la w, (\nabla^2 f(x_k) - H)w \ra. $$
Since the sequence $\{\nabla^2 f(x_k)\}$ is bounded, we have that 
$$|
\la w_k - w, \nabla^2 f(x_k) (w_k + w) \ra| \le \|w_k - w\|\cdot\|w_k + w\|. \|\nabla^2 f(x_k)\| \to 0. 
$$
Combining the latter with the the convergence $\la w_k, \nabla^2 f(x_k)w_k \ra \to \la w, Hw \ra$ brings us to
$$
\la w, (\nabla^2 f(x_k) - H)w \ra \to 0\;\mbox{ for all }\;w\in\R^n. 
$$
Choosing $w: = e_i, i = 1,\ldots,n$, tells us that all the entries on the diagonal of $\nabla^2 f(x_k) - H$ converge to $0$. Select finally $w: = e_i + e_j$ for $i, j = 1,\ldots,n$, we see that all other entries also converge to $0$ as $k\to\infty$, which therefore completes the proof of the theorem.
\end{proof}

\begin{Remark}\rm The result in \eqref{eq:hessianquad} does not hold in general if we remove the subdifferential continuity of $f$ at $\bar{x}$ for $\bar{v}$. Indeed, taking the function $f$ in Example~\ref{exam:quadandquads} gives us $\overline{\nabla}^2 f(0) =\{0,2\}$, which yields
$$
\Big\{ \frac{1}{2}q_H\;\Big|\;H \in \overline{\nabla}^2 f(\bar{x})\Big\}=\big\{q_{[0]}, q_{[1]}\big\},
$$
while we have $q_{[0]} \notin \qua f(0|0)$.
\end{Remark}

The next example shows that when $f$ is not of class $\mathcal{C}^{1,1}$, the inclusion in \eqref{eq:hessianquad} could be {\em strict}, even if $f$ is $\mathcal{C}^1$-smooth and strictly convex.

\begin{Example} Consider $f: \R \to \R$ defined by 
$$
f(x): = \left\{ \begin{array}{ll}
 x^2,  &  x < 0,\\
x^{3/2},  & x \ge 0,
\end{array} \right.
$$
with $\bar{x} = 0$ and $\nabla f(0) = 0$. Then $f$ is $\mathcal{C}^1$-smooth everywhere and strictly convex. The quadratic bundle of $f$ at $0$ for $0$ is calculated by
$$ 
\quad \qua f(0|0) =\big\{\delta_{\{0\}}, q_{[1]}\big\},
$$
and the Hessian bundle is $\overline{\nabla}^2 f(0) = \{2\}$. Therefore, \eqref{eq:hessianquad} holds as a strict inclusion in this case.
\end{Example}
\begin{proof} 
It is clear that $f$ is $\mathcal{C}^{1}$-smooth everywhere with 
$$
\nabla f(x) = f'(x) = \left\{\begin{array}{ll}
2x,  &x \le 0,  \\
3/2x^{1/2}, & x > 0,
\end{array}\right. 
$$
but not of class $\mathcal{C}^{1,1}$ around $\ox=0$ and not twice differentiable at this point. The strict convexity of $f$ follows from the fact that $f'$ is strictly increasing on $\R$. The Hessian at all nonzero points are given by
\begin{equation}\label{eq:hessian_notC11}
\nabla^2 f(x) = f''(x) =  \left\{ \begin{array}{ll}
2, & x < 0, \\
3/4x^{-1/2},  & x > 0.
\end{array} \right.
\end{equation}
To find $\overline{\nabla}^2 f(0)$, take any sequence $x_k \to 0$ such that $\nabla^2 f(x_k)$ converges to a finite number. It easily follows that $x_k < 0$ for all large $k$, since $\nabla^2 f(x_k) \to \infty$ for any $x_k \downarrow 0$. Thus the limit of $\nabla^2 f(x_k)$ is $2$. It follows from the above that $2 \in \overline{\nabla}^2 f(0)$, and we conclude in fact that $\overline{\nabla}^2 f(0) = \{2\}$.\vspace*{0.05in}

Let us now verify the generalized twice differentiability of $f$ at $0$ for $0$. We have for all $t > 0$ that
$$
\Delta_t^2 f(0|0)(w) = \dfrac{f(0 + tw) - f(0) - t0w}{\frac{1}{2}t^2} =\left\{\begin{array}{ll}
2w^2,  & w \le 0,  \\
\dfrac{2w^{3/2}}{t^{1/2}},  & w > 0. 
\end{array} \right.
$$
Passing to the limits as $t\dn 0$ gives us
$$
d^2 f(0|0)(w) = \liminf\limits_{\substack{t \downarrow 0 \\ w' \to w}} \Delta_t^2 f(0|0)(w') =  \left\{ \begin{array}{ll}
2w^2,   & w \le 0, \\
\infty,   & w > 0.
\end{array} \right.
$$
Since $d^2 f(0|0)$ is not a generalized quadratic form, $f$ is not generalized twice differentiable at $0$ for $0$. However, $f$ is twice differentiability of $f$ in the classical (and hence generalized) for all $x\ne 0$, and hence the set $\Omega_f$ defined in \eqref{eq:tapgtd} is given by
$$
\Omega_f =\big\{(x,\nabla f(x))\;\big|\;x \ne 0\big\}. 
$$
Take any $q \in \qua f(0|0)$ and find $(x_k, v_k) \xrightarrow{\Omega_f} (0,0)$ such that the sequence of $\frac{1}{2}d^2 f(x_k|v_k)$ epigraphically converges to $q$. Then we have that $x_k \to 0$ with $x_k \ne 0$ for all $k$. Consider the following two possibilities:
\begin{itemize}
\item[\textbf{(i)}] There exists a subsequence $\{x_{k_m}\}$ such that $x_{k_m} < 0$ for all $m$. Then we have the functions
$$
\frac{1}{2}d^2 f(x_{k_m}|v_{k_m}) = \frac{1}{2}\cdot q_{[2]} = q_{[1]},
$$ 
which epigraphically converge as $m\to\infty$ to the function $q_{[1]}$, and so $q = q_{[1]}$.

\item[\textbf{(ii)}] There exists no subsequence $\{x_{k_m}\}$ with $x_{k_m} < 0$ for all $m$, and so $x_k > 0$ for all large $k$. Let us show that the functions 
$$
\frac{1}{2}d^2 f(x_k|v_k)(w) = \frac{3}{4\sqrt{x_k}}w^2
$$ 
epigraphically converge to $q = \delta_{\{0\}}$. If $w \ne 0$, then any sequence $w_k \to w$ satisfies 
$$ 
\lim\limits_{k \to \infty} d^2 f(x_k|v_k)(w_k) = \lim\limits_{k \to \infty} \frac{3}{4\sqrt{x_k}}w_k^2 = \infty.
$$
Appealing to Proposition~\ref{prop:epichar} yields $q(w) =\infty$ for all $w \ne 0$, and thus
$$
q(0) = \min\big\{ \alpha \in \overline{\mathbb{\R}}\;\big|\;\exists\, w_k \to 0\;\mbox{ with }\;\liminf d^2 f(x_k|v_k)(w_k) = \alpha\big\}= 0.
$$
\end{itemize}
We get that $q_{[1]}$ and $\delta_{\{0\}}$ are all possible candidates for $q \in \qua f(0|0)$. It is easy to check that these functions indeed belong to $\qua f(0|0)$. Therefore, the formula for $\qua f(0|0)$ is verified.
\end{proof}

In the rest of this section, we demonstrate that quadratic bundles provide efficient tools to study {\em prox-regular} functions by showing that the functions from this class are always {\em twice differentiable with respect to quadratic bundles}.

\begin{Theorem}\label{prop:quadproxne} Let $f: \R^n \to \overline{\R}$ be prox-regular at $\bar{x}$ for $\bar{v} \in \partial f(\bar{x})$. Then the quadratic bundle $\qua f(\bar{x}|\bar{v})$ is nonempty.
\end{Theorem}
\begin{proof} If $f: \R^n \to \overline{\R}$ is prox-regular at $\bar{x}$ for $\bar{v} \in \partial f(\bar{x})$, then  Lemma~\ref{prop:prox-mor-pr} allows us to fix a small $\lambda > 0$ so that the associated Moreau envelope $e_{\lambda} f$ is of class $\mathcal{C}^{1,1}$ around $\bar{z} := \bar{x} + \lambda \bar{v}$. Consequently, the set $\overline{\nabla}^2 e_{\lambda} f(\bar{z})$ is nonempty. Pick any $H \in \overline{\nabla}^2 e_{\lambda} f(\bar{z})$ and find a sequence $z_k \to \bar{z}$ such that $e_{\lambda} f$ is twice differentiable at $z_k$ with $z_k \to H$ for all $k\in\N$. Define the vectors
$$
v_k := \nabla e_{\lambda} f(z_k), \quad x_k := z_k - \lambda v_k\;\mbox{ whenever }\;k\in\N.
$$
Lemma~\ref{lem:zkvkxk} tells us that $(x_k,v_k) \xrightarrow{\gph \partial f} (\bar{x},\bar{v})$ and $f(x_k) \to f(\bar{x})$ as $k\to\infty$. Furthermore, we deduce from 
Corollary~\ref{cor:gtdtwicediff}, by the twice differentiability of $e_\lambda f$ at $z_k = x_k + \lambda v_k$, that $f$ is generalized twice differentiable at $x_k$ for $v_k$ with large $k$. 
Remark~\ref{prop:extractepi}\textbf{(i)} allows us to extract a subsequence of the functions $\dfrac{1}{2}d^2 f(x_k|v_k)$ that epigraphically converges to some function $q$. Observe that we keep the convergence properties along the aforementioned subsequence: $(x_{k_m},v_{k_m}) \xrightarrow{\gph \Omega_f} (\bar{x},\bar{v})$ and $f(x_{k_m}) \to f(\bar{x})$ as $m\to\infty$. It follows from Remark~\ref{prop:extractepi}\textbf{(ii)} that $q$ is a generalized quadratic form. Finally, Definition~\ref{defi:quadbund} confirms that $q \in \qua f(\bar{x}|\bar{v})$, and so we are done with the proof of the theorem.
\end{proof}

The last example here demonstrates that the conclusion of Theorem~\ref{prop:quadproxne} fails if $f$ is merely $\mathcal{C}^1$-smooth and not prox-regular.

\begin{Example} Let $f: \R \to \R$ be given by $f(x):= -|x|^{3/2}$, $x \in \R$. Then $f$ is $\mathcal{C}^1$-smooth around $x = 0$ and $f'(0) = 0$ but $\qua f(0|0) = \emptyset$.
\end{Example}
\begin{proof}
We get that $f'(x) = -\dfrac{3}{2}\sqrt{|x|}$, that $f$ is not twice differentiable at $0$, and that $f''(x) = -\dfrac{3}{4\sqrt{|x|}}$ when $x\ne 0$. This brings us to the formula
$$
d^2 f(x|f'(x))(w) = -\dfrac{3}{4\sqrt{|x|}} w^2\;\mbox{ for all }\;x\ne 0\;\mbox{ and }\;w\in\R.
$$
Moreover, we get for $t > 0$ and $w \in \R$ that
$$
\Delta_t^2 f(0|0)(w) = \dfrac{f(0 + tw) - f(0) - 0.tw}{\frac{1}{2}t^2} = \frac{-3|w|^{3/2}}{2t^{1/2}}.
$$
It easily follows therefore that
$$
d^2 f(0|0)(w) = \liminf\limits_{\substack{t \downarrow 0\\w' \to w}} \Delta_t^2 f(0|0)(w') = -\infty, \quad w \in \R. 
$$
In this way, we arrive at the representation
$$
\Omega_f =\big\{(x,f'(x))\big|\;x \ne 0\big\}. 
$$
Assume further that $q \in \qua f(0|0)$. Then $q$ is a generalized quadratic form, and there exists a sequence $(x_k,v_k) \xrightarrow{\Omega_f} (0,0)$ such that $\dfrac{1}{2}d^2 f(x_k|v_k) \xrightarrow{e} q$; in particular, $x_k \ne 0$ for all $k$. We intend to show that there exists $w \in \R$ such that $q(w) = -\infty$. Indeed, for any $w \ne 0$ and all $w_k \to w$ it follows that
$$
\lim d^2 f(x_k|v_k)(w_k) = \lim\Big( - \dfrac{3}{4\sqrt{|x_k|}}w_k^2\Big) = -\infty.
$$
Employing Proposition~\ref{prop:epichar} brings us to $q(w) = -\infty$, which is a contradiction. Therefore, the quadratic bundle $\qua f(0|0)$ is empty.
\end{proof}

\section{Conclusions and Future Research}\label{sec:conclusion}

This paper contributes to the investigation of generalized twice differentiability and quadratic bundles of extended-real-valued functions. These recently introduced second-order notions belong to the primal-dual realm of variational analysis with great potential for applications. We reveal new properties of generalized twice differentiable functions, establish their crucial characterization via  classical twice differentiability of the associated Moreau envelopes, and apply them to the study of quadratic bundles. The obtained results on generalized twice differentiability  of extended-real-valued functions lead us to establishing the fundamental property of quadratic bundles asserting their nonemptiness for the broad class of prox-regular functions, which are the most important in second-order variational analysis.\vspace*{0.05in} 

Our future research will concentrate on applications of the developed tools of generalized twice differentiability and quadratic bundles to new primal-dual characterizations of (strong) variational convexity of extended-real-valued functions and tilt stability of local minimizers with applications to (strong) variational sufficiency in optimization and numerical algorithms of the first and second orders. We also plan to proceed with the study of infinite-dimensional extensions of the developed methods and results.

\end{document}